\documentclass[12pt]{amsproc}
\usepackage{amsthm,amssymb,amsmath,amsopn,amsfonts,etoolbox,mathptmx}
\usepackage{stmaryrd,calc,mathrsfs,enumerate}
\usepackage{color}
\usepackage{marginnote}
\usepackage{todonotes}

\usepackage{geometry}
\usepackage{prerex}

\makeatletter
\patchcmd{\@settitle}{\uppercasenonmath\@title}{}{}{}
\patchcmd{\@setauthors}{\MakeUppercase}{\scshape}{}{}
\patchcmd{\section}{\scshape}{\bfseries}{}{}
\renewcommand{\@secnumfont}{\bfseries}
\patchcmd{\abstract}{\scshape\abstractname}{\textbf{\abstractname}}{}{}
\makeatother

\usepackage{hyperref} 
\hypersetup{
	colorlinks=true,
	citecolor=blue,
	linkcolor=red
}

\newtheorem{theorem}{Theorem}[section]
\newtheorem{proposition}[theorem]{Proposition}
\newtheorem{lemma}[theorem]{Lemma}
\newtheorem{cor}[theorem]{Corollary}

\theoremstyle{definition}
\newtheorem{definition}[theorem]{Definition}
\newtheorem{ex}[theorem]{Example}

\theoremstyle{remark}
\newtheorem{remark}[theorem]{Remark}

\newcommand{\im}{\mathrm{im}\;}

\newcommand{\Id}{\mathrm{Id}}

\title{Hodge theory on tropical curves}
\author{Yury V. Eliyashev}
\address{Yu.V.~Eliyashev, National Research University Higher School of Economics, Saint Petersburg, Russia}
\email{eliyashev@gmail.com}

\begin{document}

\begin{abstract}
We construct an analog of the Hodge theory on complex manifolds on tropical curves. We use the analytical approach to the problem, it is based on language of tropical differential forms and methods of $L^2-$cohomologies. 

\end{abstract}
\maketitle
\section{Introduction}

In this paper we construct a tropical analog of the classical Hodge theory for K\"{a}hler manifolds.
We study only the case of one-dimensional tropical varieties, it is a tropical analog of the Hodge theory on smooth complex curves.

Let us very briefly recall the classical Hodge theory for complex curves. For general references on the Hodge theory see \cite{Dem}, \cite{GH}.
Given a smooth compact complex curve $C$ of genus $n.$ 
Let $\mathcal{E}^{p,q}(C)$ be the space of smooth $(p,q)-$differential forms on $C.$  
The Dolbeault cohomology $H^{p,q}_{\overline{\partial}}(C)$ is the cohomology group of the complex 
$(\mathcal{E}^{p,*}(C),\overline{\partial}),$ where $\overline{\partial}$ is a differential
$$\overline{\partial}:\mathcal{E}^{p,q}(C)\rightarrow\mathcal{E}^{p,q+1}(C).$$
Suppose $g$ is a hermitian metric on $C$ and $\omega$ is the corresponding K\"{a}hler form.
The metric $g$ induces a scalar product on $\mathcal{E}^{p,q}(C)$ and defines the Hodge star operator $$*:\mathcal{E}^{p,q}(C) \rightarrow \mathcal{E}^{1-p,1-q}(C).$$
Let $\overline{\partial}^*$ be a metric adjoin to $\overline{\partial}.$
The Laplace-Beltrami operator is defined as follows:
$$\Delta = \overline{\partial}^*\overline{\partial}+\overline{\partial}\overline{\partial}^*: \mathcal{E}^{p,q}(C) \rightarrow \mathcal{E}^{p,q}(C).$$
The space of harmonic forms $\mathcal{H}^{p,q}(C)$ is by definition the kernel of $\Delta:\mathcal{E}^{p,q}(C) \rightarrow \mathcal{E}^{p,q}(C).$

Our main goal is to prove a tropical analog of the following statement.
\begin{theorem}
Every harmonic form $\varphi\in \mathcal{H}^{p,q}(C)$ is $\overline{\partial}-$closed and, consequently, defines a cohomology class 
$[\varphi]\in H^{p,q}_{\overline{\partial}}(C)$. The map $\varphi \rightarrow [\varphi]$ is an isomorphism between 
$\mathcal{H}^{p,q}(C)$ and $H^{p,q}_{\overline{\partial}}(C)$.
The Hodge star operator is an isomorphism between the spaces of harmonics forms 
$*:\mathcal{H}^{p,q}(C)\simeq\mathcal{H}^{1-p,1-q}(C).$
The dimensions of cohomology groups are equal to
$$\dim_{\mathbb{C}} H^{0,0}_{\overline{\partial}}(C) = \dim_{\mathbb{C}} H^{1,1}_{\overline{\partial}}(C)=1,$$
$$\dim_{\mathbb{C}} H^{1,0}_{\overline{\partial}}(C) = \dim_{\mathbb{C}} H^{0,1}_{\overline{\partial}}(C)=n,$$
where $n$ is the genus of $C.$
\end{theorem}

A one-dimensional tropical variety is essentially a metric graph $\Gamma$ with some additional features.
An analog of smooth $(p,q)-$differential forms is a special class of tensor fields on edges of $\Gamma$ satisfying some boundary conditions at vertices. We denote this class of tensor by $\mathcal{E}^{p,q}(\Gamma)$ and call it the \emph{space of regular tropical superforms} of degree $(p,q).$
There are a differential $$d'':\mathcal{E}^{p,q}(\Gamma)\rightarrow \mathcal{E}^{p,q+1}(\Gamma)$$ and 
the cohomology group $H^{p,q}_{d''}(\Gamma)$ of the complex $(\mathcal{E}^{p,*},d'').$

For references on superforms and tropical cohomologies
see \cite{IKMZ}, \cite{JSS}, \cite{Lag}: notion of tropical cohomologies was introduce in \cite{IKMZ} using methods of algebraic topology, notion of tropical superforms, which play in the tropical case role of smooth differential forms, was introduced in \cite{Lag}, and in \cite{JSS} differential topological approach to tropical cohomologies was developed, i.e., an analog of the de Rham cohomology theory. In \cite{GJR} the de Rham cohomology of metric graphs is considered in a fashion quite similar to our paper.

The space $\mathcal{E}^{p,q}(\Gamma),$ the operator $d'',$ and the cohomology group $H^{p,q}_{d''}(\Gamma)$ 
play in the tropical case the same role as $\mathcal{E}^{p,q}(C), \overline{\partial}, $ and $H^{p,q}_{\overline{\partial}}(C)$ in the complex case.

There are tropical analogs of a K\"{a}hler form and a hermitian metric on $\Gamma.$
This tropical K\"{a}hler form induces a scalar product on $\mathcal{E}^{p,q}(\Gamma).$ Let $d''^*$ be the metric adjoin operator to $d''$ with respect to this scalar product.
Then we can define the Laplace-Beltrami operator 
$$\Delta = d'' d''^*+ d''^* d''.$$ The space of \emph{harmonic forms} $\mathcal{H}^{p,q}(\Gamma)$ is by definition the kernel of 
$\Delta.$

The genus of tropical curve $\Gamma$ is, by definition, the rank of $H^{1}(\Gamma),$ where $H^{1}(\Gamma)$ is the usual topological cohomology group of the graph $\Gamma.$ The main result of this paper is the following
\begin{theorem}
 Let $\Gamma$ be a tropical curve of genus $n.$ Every harmonic superform $\varphi\in \mathcal{H}^{p,q}(\Gamma)$ is $d''-$closed and, consequently, defines the cohomology class $[\varphi]\in H^{p,q}_{d''}(\Gamma)$. The map $\varphi \rightarrow [\varphi]$ is an isomorphism between $\mathcal{H}^{p,q}(\Gamma)$ and $H^{p,q}_{d''}(\Gamma).$
 The Hodge star operator maps harmonic superform to harmonic superform and the map
 $*:\mathcal{H}^{p,q}(\Gamma) \rightarrow \mathcal{H}^{1-p,1-q}(\Gamma)$ is an isomorphism. 
 There are isomorphisms $$H^{1,1}(\Gamma)\simeq H^{0,0}(\Gamma) \simeq H^{0}(\Gamma,\mathbb{R})\cong \mathbb{R}$$ and
  $$H^{1,0}(\Gamma)\simeq H^{0,1}(\Gamma) \simeq H^{1}(\Gamma,\mathbb{R}) \cong \mathbb{R}^n.$$ 
\end{theorem}

We consider our results in the first place as a toy model and a proof of a concept for the tropical Hodge theory, only then we consider as a results about topology of tropical curves. Indeed, one can compute cohomology of a tropical curve using much simpler methods without any functional analysis or differential topology. The actual problem is construct the tropical Hodge theory in higher dimensions. 
To do so one can follow the way of classical complex Hodge theory and methods of this paper, but it seems, that there are many technical obstacles in this way. The main source of these obstacles are: nonsmoothness of tropical varieties, which are locally behave like a polyhedral complexes, and complicated behavior of various analytical objects at infinity, which is typical in $L^2-$cohomology theories. Both problems are illustrated in this paper. The article \cite{Tel} on a PL-Hodge theory was a great source of inspiration for our work. 

The paper is organized as follows. In the second section we introduce the main objects and work with differential-topological part of the problem. In the third section section we develop methods related to the functional analysis and $L^2-$cohomology theory. For general references on $L^2-$methods in the complex Hodge theory see, for example, \cite[Chapter VIII]{Dem}.  

It is interesting that the topic of this paper is closely related to quantum graphs. The main idea of quantum graphs is to study the Schr\"{o}dinger equation and the Laplace equation over a metric graph \cite{BK}. In this case various boundary conditions at vertices of the graph arise. We do not know any source in the literature where methods of quantum graphs were applied to the tropical geometry.

\section{Tropical curves and tropical superforms}
In this section we introduce main objects of our paper: Tropical curves, Tropical superforms, Tropical cohomologies, tropical K\"{a}hler form and operations on them. We study their properties, and prove some results of differential-topological nature. Also, we show the relation of this objects to the complex geometry.

\subsection{Tropical curves.}
\begin{definition}
A \emph{compact connected tropical curve} $\Gamma$ is a connected metric graph with the set of vertexes $V$ and the set of edges $E$ 
satisfying the following condition:
\begin{enumerate}
 \item The sets $E$ and $G$ are finite and non-empty.
 \item The length  $l(e)$ of an edge $e\in E$ is a positive real number or $+\infty.$
 \item The length $l(e)$ is equal to $+\infty$ if and only if $e$ is incident on a degree one vertex.
 \item A finite length edge $e$ is isometric to the closed interval $[-l(e),0]$ with the standard Euclidean metric.
 \item If an infinite length edge $e$  is incident to a degree one vertex and to a vertex of a higher degree, then the edge 
 is isometric to the closed interval $[-\infty,0]$ with the standard Euclidean metric, where $-\infty$ is the image of the degree one vertex.
 \item If an infinite length edge $e$ is incident to two degree one vertices, then this edge is isometric to  $[-\infty,+\infty].$ Since $\Gamma$ is connected it have to be a graph with one edge and two vertices and the whole graph is isometric to $[-\infty,+\infty].$ 
\end{enumerate}
\end{definition}
In this paper we will address a compact connected tropical curve as just a curve.
The \emph{genus} of a curve $\Gamma$ is defined as the rank of cohomology group $H^1(\Gamma).$

\begin{ex}
Let us consider several examples of curves.
\begin{itemize}
 \item Any metric graph with finite-length edges such that all vertices has degree $\geq2$ can be considered as a tropical curve.
 \item The closed interval $[-\infty,+\infty]$ is an example of a genus $0$ curve. We can consider this curve as the tropical projective line $\mathbb{TP}^1:=[-\infty,+\infty].$ 
 \item A finite number of disjoint copies of $[-\infty,0]$ glued together along $0$
$$\Gamma=[-\infty,0]\sqcup\dots\sqcup[-\infty,0]/\sim$$
is another example of a genus $0$ curve.
\end{itemize}
\end{ex}

\begin{remark}
Usually in tropical geometry a tropical variety is defined in terms of polyhedral complexes in $\mathbb{TR}^n=[-\infty,+\infty)^n.$ One can consider these polyhedral complexes as varieties embedded to some ambient space. 
In our definition we do not use any ambient space or embeddings. 
For any compact connected tropical curve $\Gamma$ in the sense of our definition one can construct an isomorphic tropical curve in terms of polyhedral complexes in $\mathbb{TR}^n$.
\end{remark}

\begin{remark}
We can consider $S^1= \mathbb{R}/ a \mathbb{Z},$ where $a\in(0,+\infty),$ as an example of a tropical curve of genus $1.$
Informally it is a metric graph with one edge and no vertexes and it is not consistent with our definition of a curve.
To resolve this problem one can put a vertex to this curve and consider this as a metric graph with one vertex and one loop.
Another approach to deal with this problem is to extend the definition of a tropical curve using the way similar to the definition  a topological manifold in terms of charts and transition maps. According to this approach a tropical curve is a space locally isomorphic to a metric graph and transition maps are given by affine functions. We are not going use this approach in the paper.  
\end{remark}

\subsection{Tropical superforms over $\mathbb{R}$.}

\begin{definition}
A tropical superforms of degree $(p,q), p,q=0,1,$ on $\mathbb{R}$ is a smooth section of 
the line bundle 
$$\Lambda^{p,q}T^* \mathbb{R} := \bigwedge\nolimits^p T^* \mathbb{R}  \otimes \bigwedge\nolimits^q T^* \mathbb{R}.$$ We denote the linear space of $(p,q)-$tropical superforms by $\mathcal{E}^{p,q}(\mathbb{R}).$
\end{definition}
Here $\bigwedge^0 T^* \mathbb{R}$ is a trivial line bundle. In particular, $(0,0)$-tropical superforms are smooth functions on $\mathbb{R}$, 
the spaces of $(1,0)$ and $(0,1)$ tropical superforms can be identified with differential $1-$forms, i.e., with 
tensor fields of valency $(0,1),$ and $(1,1)$-tropical superforms can be identified with with 
tensor fields of valency $(0,2).$

Let $x$ be a cartesian coordinate on $\mathbb{R}.$  We denote by $d' x$ and by $d'' x$ the differential $d x$ which we consider, correspondingly, as a $(1,0)-$tropical form or $(0,1)-$tropical form, and we denote by $d' x \wedge d'' x$ the tensor field $d x \otimes d x$  which we consider as a $(1,1)-$form.
Then any $(0,0),(1,0),(0,1),(1,1)-$tropical superform can be written, correspondingly, as $\varphi(x),\varphi(x) d' x, \varphi(x) d'' x, \varphi(x) d' x \wedge d'' x$ for some smooth function $\varphi(x).$

There is a natural wedge product 
$$\wedge:\mathcal{E}^{p,q}(\mathbb{R})\otimes \mathcal{E}^{p',q'}(\mathbb{R}) \rightarrow \mathcal{E}^{p+p',q+q'}(\mathbb{R}).$$ 
The wedge product of the $(1,0)-$tropical form $d' x$ and the $(0,1)-$form $d'' x$ is defined to be equal to the  $(1,1)-$form $d' x  \wedge d'' x.$
The wedge product satisfies the alternation condition $$d' x  \wedge d'' x = - d'' x \wedge d' x,$$ $$d' x  \wedge d' x= d'' x  \wedge d'' x=0.$$

There is the differential $d'':\mathcal{E}^{p,q}(\mathbb{R}) \rightarrow \mathcal{E}^{p,q+1}(\mathbb{R}).$ It is defined on $(0,0)-$forms, i.e., functions, as
$$d''(\varphi(x))=\frac{\partial \varphi(x)}{\partial x} d'' x$$
and on $(1,0)-$forms as
$$d''(\varphi(x) d' x )= d''(\varphi(x)) \wedge d' x  = - \frac{\partial \varphi(x)}{\partial x} d'x \wedge d'' x,$$
In all other cases $d''$ is equal to zero for dimensional reasons.

In the same way we define the differential operator $$d':\mathcal{E}^{p,q}(\mathbb{R}) \rightarrow \mathcal{E}^{p+1,q}(\mathbb{R}):$$
$d'(\varphi(x))=\frac{\partial}{\partial x}\varphi(x) d' x$ and $d'(\varphi(x) d'' x )= d'(\varphi(x)) \wedge d'' x  =  \frac{\partial}{\partial x}\varphi(x) d'x \wedge d'' x,$ in all other cases $d'$ is equal to zero.

The \emph{tropical integral} over $\mathbb{R}$ of a $(1,1)$-forms $\varphi(x) d' x \wedge d'' x$ is defined as
$$\int_\mathbb{R} \varphi(x) d' x \wedge d'' x = \int_\mathbb{R} \varphi(x) d x,$$
where the right hand side is the usual integral. The integral over an iterval $I$ of $\mathbb{R}$ is defined in the same way $\int_I \varphi(x) d' x \wedge d'' x = \int_I \varphi(x) d x$. 
\begin{remark}

From an abstract point of view to define the tropical integral of $(1,1)$-tropical superform $\omega$ over a $1$-dimensional $\mathbb{R}-$linear space $L$ we need to choose a volume form $\mu$ with a constant coefficient. We can identify this form with a non-zero element of $T^*_0 L,$ since a constant section of $T^*L$ is defined by 
its value at $0.$

This tropical integral depends on the choice of $\mu,$ we denote it by $\int_{(L,\mu)} \omega.$
The form $\mu\otimes\mu$ defines a trivialization of $T^*L\otimes T^*L.$ Hence any $(1,1)$-tropical superforms $\omega$ can be written as $\omega= f(x) \mu\otimes\mu$ for some function $f(x)$ on 
$L.$ 

The integral is defined as
$$\int_{(L,\mu)} \omega = \int_{L} f(x) \mu,$$
where the right-hand side is the usual integral of a differential form of the top degree over an orientated linear space, the orientation of $L$ is 
induced by the form $\mu.$

Notice that for the form $\mu'=-\mu$ we obtain the same integral as for the form $\mu.$
Indeed, we have $\omega = f(x) \mu\otimes\mu = f(x) (-\mu)\otimes(-\mu)=f(x) \mu'\otimes\mu'$ and the integrals 
$$\int_{(L,\mu)} \omega =\int_{L} f(x) \mu =  - \int_{L} f(x) (-\mu)= - \int_{L} f(x) \mu'= \int_{\overline{L}} f(x) \mu'= \int_{(L,\mu')} \omega$$ are the same, where $\overline{L}$ is the space $L$ with an opposite orientation, i.e., the orientation induced by $\mu'$.

On the other hand, scaling of the volume form changes the value of the integral.
Indeed, if $\mu'= c \mu, c>0,$ then  $f(x) \mu\otimes\mu= \frac{1}{c^2}f(x) \mu'\otimes\mu'.$ Thus we get
$$\int_{(L,\mu)} \omega =\int_{L} f(x) \mu,$$ 
$$\int_{(L,\mu')} \omega = \int_{L} \frac{1}{c^2}f(x) \mu' = \int_{L} \frac{1}{c}f(x) \mu = \frac{1}{c} \int_{(L,\mu)} \omega.$$

Let $\Lambda$ be a lattice in $L,$ then it determines the form $\mu$ uniquely up to the sign by the condition 
$\mu(e) = 1,$ where $e$ is a generator of $\Lambda.$ Here $\mu(e)$ is a contraction of  $\mu\in T^*_0 L$ and $e\in T_0 L \simeq L.$ Therefore the lattice $\Lambda$ defines the tropical integral uniquely.

It the case $L=\mathbb{R}$ we choose $\Lambda$ to be equal to $\mathbb{Z}$ and $\mu$  to be equal to the differential of the cartesian coordinate $dx,$ this give us the initial definition of the tropical integral.
\end{remark}

\subsection{Tropical superforms over tropical curve}
Let $e$ be an edge of $\Gamma$ then the space $\mathcal{E}^{p,q}(e)$ of $(p,q)-$tropical superforms over $e$ is defined as the restriction of $\mathcal{E}^{p,q}(\mathbb{R})$ to either 
$[-l(e),0]$ or $(-\infty,0]$ or $(-\infty,+\infty)$ if $e$ is, consequently, isometric to $[-l(e),0]$ or $[-\infty,0]$ or $[-\infty,+\infty].$
An integral $\int_e \omega$ over the edge $e$ of a form $\omega\in\mathcal{E}^{1,1}(e)$ is defined as an tropical integral over the corresponding interval 
of $\mathbb{R}.$

\begin{definition}
The linear space $\widetilde{\mathcal{E}}^{p,q}(\Gamma)$ \emph{of tropical superforms of degree $(p,q),$ $p,q = 0,1$ on a curve $\Gamma$} is defined as 
follows  
$$\widetilde{\mathcal{E}}^{p,q}(\Gamma)=\bigoplus_{e\in E} \mathcal{E}^{p,q}(e).$$ We denote by $\omega_e$ the component over the edge $e$ of the form $\omega\in\widetilde{\mathcal{E}}^{p,q}(\Gamma).$\end{definition}

The integral of a form $\omega\in \widetilde{\mathcal{E}}^{1,1}(\Gamma)$ over $\Gamma$ is defined as the sum of the tropical integrals over all edges:
$$\int_{\Gamma}\omega=\sum_{e\in E}\int_{e}\omega_e.$$

Notice that for a form from $\widetilde{\mathcal{E}}^{p,q}(\Gamma)$ there are no conditions over values of the form at the ends of different edges which represent the same vertex of $\Gamma.$ Actually, this space is not an adequate analog of the smooth differential form on a Riemann surface and will play supplementary role in the paper. The tropical analog of smooth forms is \emph{regular tropical superforms} which is defined below.

\begin{definition}\label{def.reg}
The \emph{space of regular tropical superforms} $\mathcal{E}^{p,q}(\Gamma)$ is a subspace of $\widetilde{\mathcal{E}}^{p,q}(\Gamma).$ Elements of $\mathcal{E}^{p,q}(\Gamma)$ 
should satisfy the following conditions:
\begin{enumerate}
 \item {\bf Continuity.} A $(0,0)$-form $\varphi\in \mathcal{E}^{0,0}(\Gamma)$ is continuous if for any two edges $e,e'\in E$ incident to the same vertex $v,$
 the values of $\varphi_e$ and $\varphi_{e'}$ at the points corresponding to $v$ coincide. In the other words, $\varphi$ have to be a continuous function on the metric graph $\Gamma.$ 
 \item {\bf Kirchhoff's law.} Given a vertex $v\in V.$ Let $E_v\subset E$ be a set of edges incident to this vertex.
 Suppose that an edge $e\in E_v$ is identified with the interval $[-l(e),0]$ and the point $0\in[-l(e),0]$ corresponds to the vertex $v.$
 Let $\varphi\in \widetilde{\mathcal{E}}^{1,0}(\Gamma)$ be a $(1,0)$-form, $\varphi_e=\varphi_e(x) d'x.$
 We say that the form $\varphi$ satisfies Kirchhoff's law at the vertex $v$ if $$\sum_{e\in E_v} \varphi_e(0)=0.$$
 The form $\varphi$ satisfies Kirchhoff's law on the curve $\Gamma$ if it satisfies Kirchhoff's law at every vertex  of $\Gamma$ of degree $\geq 2.$
\item {\bf Regularity at infinity.} We say that a superform $\omega$ is regular at infinity if for any degree one vertex $v$ of $\Gamma$ there is a neighborhood $U$ of $v$ such that the restriction of $\omega$ to $U$
is a constant function if $\omega$ has degree $(0,0)$ and identically equal to zero otherwise, i.e., if $\omega$ has degrees $(1,0),(0,1),(1,1).$
\end{enumerate}
\end{definition}
Thus a form is regular if is regular at infinity and, in addition to this,
is continuous in case of $(0,0)-$forms, or satisfies the Kirchhoff's law in the case of $(1,0)-$form.

\begin{proposition}
 The space of regular tropical superforms $\mathcal{E}^{*,*}(\Gamma)$ is closed under the wedge product and the $d''$-differential.
\end{proposition}
The proof is straightforward.  

\begin{theorem}[Stokes' theorem]\label{th.stokes}
If $\omega\in \mathcal{E}^{1,0}(\Gamma),$ then
$\int_\Gamma d'' \omega=0.$
Consequently, if $\varphi\in \mathcal{E}^{p,0}(\Gamma)$ and $\psi\in \mathcal{E}^{1-p,0}(\Gamma),$ then
$$\int_\Gamma d'' \varphi\wedge \psi = (-1)^{p+1} \int_\Gamma \varphi\wedge d''\psi.$$
\end{theorem}
\begin{proof}
Let $\omega$ be an element of $\mathcal{E}^{1,0}(\Gamma).$
Let $e$ be an edge of $\Gamma$ and  $\omega_e(x) d'x$ be a restriction of $\omega$ to $e.$
Using the Newton-Leibniz formula we get $\int_{e} d'' ( \omega_e(x) d'x)= - \omega_e(0) + \omega_e(l(e)).$
Since the integral over $\Gamma$ is a sum of integrals over edges, combining the Newton-Leibniz formula, the Kirchhoff's law, and Regularity at infinity we obtain the first statement. 

The second statement follows from the first statement and the Leibniz's rule. 
\end{proof}

\subsection{Tropical cohomologies}
Let $U$ be an open set in $\Gamma.$ 
We define $\mathcal{E}^{p,q}_\Gamma(U)$ as a linear space of smooth $(p,q)-$forms on $U$ regular in the sense of Definition
\ref{def.reg}. The correspondence $U \rightarrow \mathcal{E}^{p,q}_\Gamma(U)$ defines the sheaf $\mathcal{E}^{p,q}_\Gamma$ of smooth tropical regular superforms over $\Gamma,$  the space $\mathcal{E}^{p,q}_\Gamma(U)$ is the space of sections of $\mathcal{E}^{p,q}_\Gamma$ over $U.$
 
Let $\Omega^1_\Gamma$ be the subsheaf of $d''-$closed forms the sheaf $\mathcal{E}^{1,0}_\Gamma.$ 
Let us notice that $(1,0)-$form $\varphi$ is $d''-$closed if its restriction to an edge $e$ $$\varphi_e(x) d'x$$
has a locally constant coefficient $\varphi_e(x).$

Let us describe the sheaf $\Omega^1_\Gamma$ more explicitly. Given a vertex $v$ of $\Gamma$ of degree $d\geq 2.$ Consider a small $\varepsilon-$neighborhood $U_\varepsilon$ of $v.$ It is isometric to  
\begin{equation}\label{eq.nbhd}
 U_\varepsilon =\bigsqcup_{\mbox{$d$-times}} (-\varepsilon,0]/\sim,
\end{equation}
where points $0$ of different intervals are identified by the equivalence relation. The equivalence class of $0$ is identified with the vertex $v$. 
A section of $\Omega^1_\Gamma$ over $U_\varepsilon$ is a collection of $(1,0)-$forms with constant coefficient 
$$\varphi_j d'x, \; \varphi_j\in\mathbb{R}, \; j=1,\dots,d,$$ where the form $\varphi_j d'x$ is defined on the $j-$th interval $(-\varepsilon,0]\subset U_\varepsilon,$  and the coefficients satisfy the Kirchhoff's law:
$$\sum^d_{j=1} \varphi_j=0.$$

Suppose $v$ is a degree $1$ vertex its $\varepsilon-$neighborhood $U_\varepsilon$ is isometric to $U_\varepsilon=[-\infty,-\varepsilon),$ since section of $\Omega^1_\Gamma$ are regular at infinity, this section are identically equal to zero on $U_\varepsilon.$

Let $\mathbb{R}_\Gamma$ be a subsheaf of locally constant functions on $\Gamma$ of the sheaf $\mathcal{E}^{0,0}_\Gamma.$
Sections of $\mathbb{R}_\Gamma$ are locally constant functions over edges satisfying the continuity property at vertices. Obviously, the subsheaf $\mathbb{R}_\Gamma$  coincides with the subsheaf of $d''-$closed functions the sheaf $\mathcal{E}^{0,0}_\Gamma.$

\begin{remark}
 The sheaves  $\mathbb{R}_\Gamma, \Omega^1_\Gamma, \mathcal{E}^{p,q}_\Gamma$ play in the tropical theory the same role as,
 correspondingly, the sheaves of holomorphic functions $\mathcal{O}_C,$ holomorphic $1-$forms $\Omega^1_C,$ and smooth $(p,q)-$differential forms $\mathcal{E}^{p,q}_C$ on a smooth curve $C$ in the complex case. The differential $d''$ play the same role as the $\overline{\partial}$ operator.
 
 The space $\Omega^1_\Gamma$ was introduced in \cite[Definition 2.25, Tropical 1-form]{Lan}. In that paper it is related to degeneration of complex curves to tropical, and related degeneration of holomorphic forms on curves. 
\end{remark}

\begin{proposition}\label{pr.ex_smooth}
There are exact sequences of sheaves 
$$0 \rightarrow \Omega^1_\Gamma \xrightarrow{i} \mathcal{E}^{1,0}_\Gamma \xrightarrow{d''}  \mathcal{E}^{1,1}_\Gamma \rightarrow 0,$$
$$0 \rightarrow  \mathbb{R}_\Gamma \xrightarrow{i}  \mathcal{E}^{0,0}_\Gamma \xrightarrow{d''}  \mathcal{E}^{0,1}_\Gamma\rightarrow 0,$$
where $i$ is the natural inclusion of subsheaves.
\end{proposition}
\begin{proof}
The map $i$ is injective by definition.
The kernel of $d''$ consists of forms with coefficients constant on edges. 
These are 
exactly forms either from $\Omega^1_\Gamma$ or from $\mathbb{R}_\Gamma.$ Therefore $\im i = \ker d''.$

The surjectivity of $d''$ follows from the Newton-Leibniz formula. 
Let $U_\varepsilon$ be an $\varepsilon-$neighborhood of a vertex $v$ as defined in (\ref{eq.nbhd}).
Given a $(1,1)-$form $\omega.$ Let $$\omega_j(x) d'x\wedge d'' x$$ be a component of $\omega$
over the $j-$th edge of $U_\varepsilon.$
Let $\varphi_j(x) =  \int^0_x \omega_j(t) d t$ and $\varphi_j(x) d'x$ be a component of the $(1,0)-$form $\varphi$ over $j-$th edge. The form $\varphi$ is regular. Indeed, $\varphi_j(0)=0$, therefore it satisfies the Kirchhoff's law at $v.$
We have $d''\varphi=\omega.$

Let $U_\varepsilon=[-\infty,-\varepsilon)$ be an $\varepsilon-$neighborhood of a degree $1$ vertex and $\omega(x) d'x\wedge d'' x$
be a regular $(1,1)-$form on $U_\varepsilon.$
Suppose $$\varphi(x)=-\int^x_{-\infty} \omega(t) dt$$ then $\varphi(x) d'x$ is a regular form on $U_\varepsilon.$
Indeed, since $\omega$ is regular at infinity it is zero at some neighborhood of $-\infty.$
Hence the integral is convergent and $\varphi(x)$ is equal to zero at the same neighborhood of $-\infty.$

For an interior point of an edge there is a neighborhood isometric to a bounded interval $U_\varepsilon=(-\varepsilon,\varepsilon).$ The exactness of sequences  
over this neighborhood follows from the Newton-Leibniz formula. 
Thus we checked all possible cases and proved that in a neighborhood of any point $x\in\Gamma$ the operator $d''$ is surjective.

The proof of the surjectivity of $d''$ in the second sequences repeats the above arguments.
\end{proof}

Let us define the bigraded cohomology group $H^{p,q}(\Gamma)$ of $\Gamma$ as
$$H^{1,q}(\Gamma)=H^{q}(\Gamma,\Omega^1_\Gamma),$$
$$H^{0,q}(\Gamma)= H^{q}(\Gamma,\mathbb{R}_\Gamma).$$
Since $\mathbb{R}_\Gamma$ is the sheaf of locally constant functions, the group $H^{0,q}(\Gamma)$ is isomorphic to the usual topological cohomology group $H^{q}(\Gamma,\mathbb{R})$ of the graph $\Gamma.$ 

\begin{proposition}\label{pr.smooth_CdR}
The sheaves $\mathcal{E}^{p,q}_\Gamma$ are fine and acyclic.
There is an isomorphism
$H^{p,q}(\Gamma)\cong H^{q}(\mathcal{E}^{p,*}(\Gamma),d''),$
where $H^{q}(\mathcal{E}^{p,*}(\Gamma),d'')$ is the cohomology group of the complex
$$0\rightarrow  \mathcal{E}^{p,0}(\Gamma) \xrightarrow{d''}  \mathcal{E}^{p,1}(\Gamma)\rightarrow 0.$$
\end{proposition}
\begin{proof}
The proof of this statement repeats the proof of acyclicity of the sheaf of smooth forms on a smooth manifold and the \v{C}ech to de Rham isomorphism on a smooth manifold.

For any open cover $\mathfrak{U}$ of $\Gamma$ there is a smooth partition of unity for the sheaf of regular tropical  $(0,0)$-superforms 
$\mathcal{E}^{0,0}_\Gamma.$ 
 Since the sheaf $\mathcal{E}^{p,q}_\Gamma$ is an $\mathcal{E}^{0,0}_\Gamma-$module, there is a partition of unity 
 on it and $\mathcal{E}^{p,q}_\Gamma$ is a fine sheaf and, consequently,  is acyclic.
 
Let $\mathfrak{U}=\{U_i\}$ be a finite acyclic open cover of $\Gamma,$ i.e., $\mathfrak{U}$ such a cover that for any intersection $U$ of elements of $\mathfrak{U}$ the sequences of section corresponding to the sequences (\ref{pr.ex_smooth}) of sheaves are exact.
Using the standard construction of the \v{C}ech to de Rham isomorphism we prove the proposition.
\end{proof}

\subsection{K\"{a}hler form, inner product and Hodge star operator.}\label{ss.inner_prod}
Let $g=g(x)  d' x \wedge d''x\in \mathcal{E}^{1,1}(\mathbb{R})$ be a positive tropical $(1,1)-$superform over $\mathbb{R}.$
We say that form is positive if $g(x)>0$ for every $x\in \mathbb{R}.$
Since  $g(x) d' x \wedge d''x$ stands for $g(x) d x \otimes d x,$ we can consider $g$ as a Riemannian metric on $\mathbb{R}.$
The Riemannian metric $g$ defines the pointwise scalar product $(\varphi,\psi)_g(x)$ between elements of $\varphi,\psi\in \mathcal{E}^{p,q}(\mathbb{R}).$
Indeed, we can consider elements of $\mathcal{E}^{p,q}(\mathbb{R})$ as tensor fields, a Riemannian metric defines the pointwise scalar product on tensor fields.
Let us define the scalar product $(\varphi,\psi)_g$ between two forms $\varphi,\psi\in\mathcal{E}^{p,q}(\mathbb{R})$ as 
$$(\varphi,\psi)_g=\int_\mathbb{R}(\varphi,\psi)_g(x) g,$$
where $g$ is consider as a tropical $(1,1)-$form and the right-hand side is a tropical integral. 
At this moment we are not concerned with convergence of this integral.
Usually we will omit subscript in $(\cdot,\cdot)_g$ and write $(\cdot,\cdot)$ instead.

Let us describe the scalar product in coordinate terms for the various $p,q:$ 
$$(f(x),h(x))=\int_{\mathbb{R}}g(x) f(x) h(x)  d x$$
$$(f(x)d'x,h(x)d'x)=(f(x)d''x,h(x)d''x)=\int_{\mathbb{R}} f(x) h(x) d x$$
$$(f(x)d' x \wedge d''x,h(x)d' x \wedge d''x)=\int_{\mathbb{R}}\frac{1}{g(x)}   f(x) h(x) d x$$

The Hodge star operator 
$$*_g: \mathcal{E}^{p,q}(\mathbb{R}) \rightarrow \mathcal{E}^{1-p,1-q}(\mathbb{R})$$
is defined by the relation 
$$ \int_\mathbb{R} \varphi \wedge *_g \psi=(\varphi,\psi)_g$$
for every $\varphi, \psi \in \mathcal{E}^{p,q}(\mathbb{R}).$
Usually we will omit the subscript in $*_g$ and write $*$ instead.

The Hodge star is an isometry, i.e., for any $\varphi,\psi\in\mathcal{E}^{p,q}(\mathbb{R})$ holds $$(\varphi,\psi)=(*\varphi,*\psi).$$
Also, for any $\psi\in\mathcal{E}^{p,q}(\mathbb{R})$  holds $$** \psi= (-1)^{p+q}\psi.$$ 
In terms of coordinate  the Hodge star looks as follows:
\begin{equation}\label{eq.hodge}
 \begin{split}
  *f(x) = f(x) g(x) d'x\wedge d''x, \\ *f(x) d'x\wedge d''x =  \frac{1}{g(x)}f(x),\\
  *f(x) d'x =   f(x) d''x,\\ *f(x) d''x = -  f(x) d'x.
 \end{split}
\end{equation}

\begin{remark}
In the differential geometry a Riemannian metric defines the standard scalar product on the space of sections of tensor fields. 
Since the form $g$ is a symmetric tensor field of valence $(0,2)$ we can consider it as a Riemannian metric on $\mathbb{R}.$ Also, we can consider the space $\mathcal{E}^{p,q}(\mathbb{R})$ as a space of tensor fields. Therefore, this Riemannian metric induces the standard scalar product on the space $\mathcal{E}^{p,q}(\mathbb{R}),$ but this scalar product does not coincide with the tropical scalar product defined above. 

Indeed, let $f(x),h(x)$ be functions on $\mathbb{R}$ then the standard scalar product on the space of functions equals
$$(f(x),h(x))=\int_{\mathbb{R}} f(x) h(x) \sqrt{g(x)} d x.$$ On the space of $1-$forms which can be identified with $\mathcal{E}^{1,0}(\mathbb{R})$ or $\mathcal{E}^{0,1}(\mathbb{R})$ the standard scalar is equal to
$$(f(x) d x,h(x) d x)=\int_{\mathbb{R}} f(x) h(x) \frac{1}{\sqrt{g(x)}} d x.$$ The reason for this is that the tropical superforms over $\mathbb{R}$ correspond to the usual differential form over $\mathbb{C}\setminus\{0\},$ not on $\mathbb{R},$ and the scalar product on the space of tropical superforms is 
consistent with the standard scalar product on $\mathbb{C}\setminus\{0\}.$ This correspondence is described in the next subsection.
\end{remark}

If we identify an edge $e$ of $\Gamma$ with an interval of $\mathbb{R},$ then a K\"{a}hler form on this interval defines 
the scalar product $(,)$ and the Hodge star operator on this edge of $\Gamma.$

\begin{definition}\label{def.kh}
A \emph{K\"{a}hler form} $g$ on the curve $\Gamma$ is a $(1,1)-$form $g\in\widetilde{\mathcal{E}}^{p,q}(\Gamma)$
such that 
\begin{enumerate}
 \item $g$ is positive, i.e., in local coordinates it is given by  $g= g(x) d' x \wedge d''x$ with positive $g(x);$
 \item $\int_\Gamma g < + \infty;$
 \item on any infinite length edge $e$ the integral $\int_e x^2 g(x) d' x \wedge d''x$ converges.
\end{enumerate}
\end{definition}
We define the scalar product for $\varphi,\psi \in \widetilde{\mathcal{E}}^{p,q}(\Gamma)$ as follows:
 $$(\varphi,\psi)_g=\int_\Gamma \varphi \wedge *_g \psi.$$ 
\begin{remark}
The last condition in the definition, convergence of $\int_e x^2 g(x) d' x \wedge d''x$, play its role in the study of $L^2-$theory in the next section. It allow us to get some estimates on convergence of various integrals. It is not clear for us what is necessary and sufficient condition here or how this condition can be weakened. 

As Example \ref{ex.FS} shows tropical K\"{a}hler forms that arises from the complex geometry are, actually,
have rapidly decreasing at infinity coefficients, which is, actually, much higher rate of convergence that we required in the definition.
\end{remark}

\begin{remark}
In the tropical setting a K\"{a}hler form $g$ plays the role of a K\"{a}hler form and a hermitian metric in the classical complex geometry. 
Also the K\"{a}hler form $g$ defines a Riemannian metric on each edge of $\Gamma.$ 
Notice that, this Riemannian metric is unrelated to the metric structure on $\Gamma,$ i.e., to the length of the edges.
In general, the Hodge star operator does not preserve the regularity conditions, i.e., there is a regular form $\varphi\in \mathcal{E}^{p,q}(\Gamma)$ such that
$* \varphi$ is not regular. This is important and unfortunate difference between the tropical and the classical settings. Indeed, the Hodge star of a smooth from on a manifold is again a smooth form.   
\end{remark}

We can summarize the results of this section as follows
\begin{theorem}
Let $g$ be a \emph{K\"{a}hler form} on the curve $\Gamma.$ 
Then $\mathcal{E}^{p,q}(\Gamma)$ is a differential bigraded algebra  with the nondegenerate pairing
$$\langle\cdot,\cdot\rangle:\mathcal{E}^{p,q}(\Gamma)\otimes \mathcal{E}^{1-p,1-q}(\Gamma) \rightarrow \mathbb{R},$$
$$\langle\varphi,\psi\rangle=\int_\Gamma \varphi\wedge \psi,$$
and the scalar product
$$(\varphi,\psi)= \int_\Gamma \varphi\wedge *\psi.$$
\end{theorem}
In this theorem the scalar product and and pairing are well-defined for all elements, i.e., all integrals are convergent.
 The convergence follows from the regularity at infinity condition for regular forms and the convergence of the tropical integral $\int_\Gamma g$
 which is required by the definition of the K\"{a}hler form $g.$

\subsection{Tropical superforms and complex geomtery}
At the first glance the tropical superforms and related objects may seem to be a bit artificial constructions. 
In this section we show that these objects can be naturally interpreted in terms of the classical complex geometry. 

The real line $\mathbb{R}$ can be considered as a tropical analog of the complex torus $\mathbb{C}^*.$
There is the map $$\log|z|:\mathbb{C}^*\rightarrow \mathbb{R}$$ between them. 

Let $\mathcal{E}^{p,q}(\mathbb{C}^*)$ be a space of smooth differential $\mathbb{C}$-valued forms of bidegree $(p,q)$ over  $\mathbb{C}^*.$
Let us define the bigraded algebra homomorphism $\Theta:\mathcal{E}^{p,q}(\mathbb{R}) \rightarrow \mathcal{E}^{p,q}(\mathbb{C}^*).$ 
On the generators it is defined as follows
$$\Theta(\varphi(x)) = \varphi(\log|z|),\; \varphi(x)\in \mathcal{E}^{0,0}(\mathbb{R}),$$
$$\Theta(d' x) = \frac{1}{2\sqrt{\pi}} \frac{d z}{z},\; \Theta(d'' x) = \frac{i}{2\sqrt{\pi}} \frac{d \overline{z}}{\overline{z}}.$$
Since $\mathcal{E}^{p,q}(\mathbb{R})$ is an $\mathbb{R}-$algebra, we consider $\Theta$ as an $\mathbb{R}-$algebra homomorphism.
 
Consider the unitary group $U(1)=\{t\in\mathbb{C}:|t|=1\}$ and the standard action $U(1)\times\mathbb{C}^* \rightarrow \mathbb{C}^*,$ i.e.,  $(t,z)\rightarrow t\cdot z.$ This action induces action of $U(1)$ on $\mathcal{E}^{p,q}(\mathbb{C}^*).$ 
The image of $\mathcal{E}^{p,q}(\mathbb{R})$ under $\Theta$ is an $\mathbb{R}$-linear subspace in the space $\mathcal{E}^{p,q}_{U(1)}(\mathbb{C}^*)$ of $U(1)-$invariant forms. Its complexification $\Theta (\mathcal{E}^{p,q}(\mathbb{R}))\otimes_\mathbb{R} \mathbb{C}$ coincides with $\mathcal{E}^{p,q}_{U(1)}(\mathbb{C}^*).$

Let $g(x) d'x\wedge d''x$ be a tropical K\"{a}hler form. 
One can check that its image $$\Theta(g(x) d'x\wedge d''x)=\frac{i}{4 \pi} g(\log|z|) \frac{ d z\wedge d \overline{z}}{|z|^2}$$ 
is a K\"{a}hler form on $\mathbb{C}^*.$ A K\"{a}hler form $\omega=\frac{i}{2} h(z)d z\wedge  d \overline{z}$
determines the hermitian metric $h=h(z)d z \otimes d \overline{z}$ on $\mathbb{C}^*.$
Thus for the K\"{a}hler form
$\omega=\Theta(g(x) d'x\wedge d''x)$ the corresponding hermitian metric is
$$h=\frac{1}{2 \pi} \frac{g(\log|z|)}{|z|^2}d z \otimes d \overline{z}.$$

Let $*_h$ be the Hodge star operator and $(,)_h$ be the scalar product on $\mathcal{E}^{p,q}(\mathbb{C}^*)$
corresponding to the metric $h.$

\begin{proposition}
  Suppose $\varphi, \psi \in \mathcal{E}^{p,q}(\mathbb{R}),$ then the following relations hold:
$$*_h \Theta = \Theta *_g,$$
$$ (\Theta \varphi, \Theta\psi)_h = (\varphi,\psi)_g,$$
$$\Theta(d''\varphi)=\frac{i}{\sqrt{\pi}} \overline{\partial} \Theta(\varphi),$$
$$\Theta(d'\varphi)=\frac{1}{\sqrt{\pi}} \partial \Theta(\varphi).$$
Let $\omega \in \mathcal{E}^{1,1}(\mathbb{R}),$ then
$$\int_\mathbb{R} \omega = \int_{\mathbb{C}^*} \Theta(\omega).$$
The tropical integral $\int_I \omega$ over an interval $I=(a,b)$  is equal to 
the integral of $\int_U \Theta(\omega)$ over the annulus $U=\{z\in \mathbb{C}^*: e^a<|z|<e^b\}.$
\end{proposition}
The proof is straightforward computations.

Thus tropical superform can be reinterpreted as an $\mathbb{R}-$subalgebra of $U(1)-$invariant forms of the algebra $\mathcal{E}^{p,q}(\mathbb{C}^*).$

\begin{ex}\label{ex.FS}

Let us consider the Fubini-Study metric and its K\"{a}hler form $$\omega=\frac{i}{2\pi} \partial \overline{\partial} \log(1+|z|^2)=\frac{i}{2\pi} \frac{d z \wedge d \overline{z}}{(1+|z|^2)^2}$$
on $\mathbb{C}^*\subset \mathbb{CP}^1.$

There is the tropical form $$\omega'=2\frac{e^{2x}}{(1+e^{2x})^2} d'x \wedge d'' x$$ such that $\Theta(\omega')=\omega.$
Moreover, $\omega' = \frac{1}{2} d' d''\log(1+e^{2x})$ thus  
$$\omega=\Theta(\frac{1}{2} d' d''\log(1+e^{2x}))=\frac{1}{2} \frac{1}{\sqrt{\pi}} \partial \frac{i}{\sqrt{\pi}}\overline{\partial} \log(1+|z|^2).$$
Since $\omega'$ satisfies all condition of Definition \ref{def.kh}, we can consider $\omega'$ as a K\"{a}hler from on the tropical projective space $\mathbb{TP}^1=[-\infty,+\infty].$
\end{ex}

\begin{remark}
The tropical form $\omega'$ from the example above is not a regular tropical form according to our definition of regularity since the
regularity at infinity condition does not hold. On the other hand, the form $\Theta \omega'$ can be extended to a smooth form on $\mathbb{CP}^1.$
Let us also notice that the coefficient $\frac{e^{2x}}{(1+e^{2x})^2}$ of the form $\omega'$ is a rapidly decreasing function on $\mathbb{R}$ in the sense of Schwartz space. 

Moreover, since the coefficients of any K\"{a}hler form $g$ a curve $\Gamma$ are everywhere positive, a K\"{a}hler form $g$ fails to be regular at infinity if there are infinite length edges on $\Gamma$. 
On the other hand, its coefficients should decrease fast enough near infinity since we require the convergence of the integral $\int_\Gamma g$.

Also, notice that the map $$\log|z|:\mathbb{C}^*\rightarrow \mathbb{R}$$
can be extend to the map $$\log|z|:\mathbb{C}\rightarrow \mathbb{R}\cup \{-\infty\}$$

These observations lead us to the idea to extend the notion of regularity at infinity as follows. We may call a tropical form $\varphi\in \mathcal{E}^{p,q}(\mathbb{R})$ regular at infinity if $\Theta \varphi \in \mathcal{E}^{p,q}(\mathbb{C}^*)$ can be extended to a smooth form over whole $\mathbb{C}.$ This extension seems to be consistent but we do not develop this idea further in this paper.

\end{remark}

\section{$L^2$-theory, Laplace-Beltrami operator and harmonic form}
In this section we prove the main statements of the paper. We introduce the notions of tropical superforms with $L^2$-coefficients,
weak $d''-$differential, the Laplace-Beltrami operator, and harmonic tropical superforms. Main methods of this parts are in style of $L^2-$cohomology theory: functional analysis, unbounded differential operators, distributions, Sobolev spaces, various analytical estimations. Work with the infinite length edges require to use some tedious analysis.

\subsection{Tropical superforms with $L^2$-coefficients and weak $d''-$differential.}
Let $\Gamma$ be a tropical curve with a  K\"{a}hler from $g.$
Let us denote by $\mathcal{L}^{p,q}(\Gamma)$ the Hilbert space of $(p,q)-$form on $\Gamma$ with $L^2$-coefficients with the scalar product $(\cdot,\cdot)_g$ defined in the subsection \ref{ss.inner_prod}. This space is the metric completion of $\widetilde{\mathcal{E}}^{p,q}(\Gamma).$
Obviously, the space of regular form $\mathcal{E}^{p,q}(\Gamma)$ is a subspace of $\mathcal{L}^{p,q}(\Gamma),$ and $\mathcal{L}^{p,q}(\Gamma)$ is 
also the metric completion of $\mathcal{E}^{p,q}(\Gamma).$

There is a continuous linear extension of the Hodge star operator $*$ from $\widetilde{\mathcal{E}}^{p,q}(\Gamma)$ to $\mathcal{L}^{p,q}(\Gamma)$ which we also denote by $*.$ The Hodge star operator is an isometry between  $\mathcal{L}^{p,q}(\Gamma)$ and  $\mathcal{L}^{1-p,1-q}(\Gamma).$

\begin{definition}
A form $\omega\in\mathcal{L}^{p,1}(\Gamma)$ is called the \emph{weak $d''-$differential} of a form 
$\psi\in\mathcal{L}^{p,0}(\Gamma)$ if for any regular form $\varphi\in \mathcal{E}^{1-p,0}(\Gamma)$
holds 
\begin{equation}\label{eq.int_by_parts}
 \int_\Gamma \omega\wedge \varphi = (-1)^{p+1}\int_\Gamma \psi\wedge d'' \varphi.
\end{equation}
We denote the weak $d''-$differential of a form $\psi$ by $d''\psi.$ 
\end{definition}
Obviously, the $d''-$differential
of a regular form is also the weak $d''-$differential. 

The weak $d''-$differential of a from $\psi$ is unique if it exists.
Indeed, suppose there is two such differentials $\omega_1$ and $\omega_2.$
Then using (\ref{eq.int_by_parts}) we obtain
 $$\int_\Gamma (\omega_1-\omega_2) \wedge \varphi = (-1)^{p+1}\int_\Gamma (\psi-\psi)\wedge d'' \varphi=0$$.
 $$\int_\Gamma (\omega_1-\omega_2) \wedge \varphi=\pm \int_\Gamma ** (\omega_1-\omega_2) \wedge \varphi = \pm(*(\omega_1-\omega_2),\varphi)=0$$
 Since it holds for any $\varphi\in \mathcal{E}^{1-p,0}(\Gamma)$ and $\mathcal{E}^{1-p,0}(\Gamma)$ is dense in $\mathcal{L}^{1-p,0}(\Gamma),$ we get the equality $\omega_1=\omega_2.$

Thus, there is the densely defined unbounded operator 
$$d'':\mathcal{L}^{p,0}(\Gamma)\rightarrow \mathcal{L}^{p,1}(\Gamma).$$
We denote its domain by $\mathcal{D}(d'')$ or by $\mathcal{D}^{p,0}(\Gamma).$

In the sequel we denote the domain of an unbounded operator $A$ by $\mathcal{D}(A).$ 

\subsection{Presheaves $\mathcal{L}^{p,q}$ and $\mathcal{D}^{p,0}$.}
\begin{definition}
Restrictions of $\mathcal{L}^{p,q}(\Gamma)$ to an open subsets of $\Gamma$ defines the presheaf $\mathcal{L}^{p,q}$ of $(p,q)-$superform with $L^2$-coefficients on $\Gamma.$ Let us define the subpresheaf $\mathcal{D}^{p,0}$ of the presheaf $\mathcal{L}^{p,0}.$ For an open subset $U\subset \Gamma$ the $(p,0)-$form $ \psi\in \mathcal{L}^{p,0}(U)$ belongs to $\mathcal{D}^{p,0}(U)$ if there is $\omega \in \mathcal{L}^{p,1}(U)$ such that for any $\varphi\in \mathcal{E}^{1-p,0}(\Gamma)$ with a compact support on $U$
holds the equation (\ref{eq.int_by_parts}).
\end{definition}
\begin{ex}
 In the definition above the from  $\varphi$ has a compact support in the open set $U.$
 Let us clarify the structure of topology on the infinite edges and what is considered to be a compact support in that case.
 Consider the tropical projective space $\mathbb{TP}^1= [-\infty,+\infty].$ Then, for example, $U = [-\infty,\infty)$ is an open subset of $\mathbb{TP}^1,$ the set $[-\infty,a],a\in \mathbb{R}$ is a compact subset of $U$ and
 the set  $[a,+\infty),a\in \mathbb{R}$ is not compact.
\end{ex}

\begin{remark}
We should warn that the sheafification of  $\mathcal{L}^{p,q}$ is the sheaf $\mathcal{L}_{loc}^{p,q}$ of $(p,q)-$superform with 
locally $L^2$-coefficients on $\Gamma,$ i.e., sections of $\mathcal{L}_{loc}^{p,q}(U)$ over an open set $U$ are $(p,q)-$superform such that their coefficients are $L^2$-integrable functions over every compact set of $U.$ Since $\Gamma$ is compact, 
we have $\mathcal{L}_{loc}^{p,q}(\Gamma)=\mathcal{L}^{p,q}(\Gamma).$
To avoid complications related to the locally $L^2$-coefficients we will work with the presheaf $\mathcal{L}^{p,q}.$ 

In the other hand, $\mathcal{L}^{p,q}$ and $\mathcal{D}^{p,0}$ are almost sheaves, to be sheaves they have to satisfy  Locality and Gluing axioms. Let us recall these axioms for a sheaf $\mathcal{F}.$

(Locality) Suppose $U$ is an open set, $\{U_{i}\}_{i\in I}$ is an open cover of $U$, and $s,t\in \mathcal{F}(U)$ are sections. If $s|_{U_{i}}=t|_{U_{i}}$ for all $i\in I$, then $s=t$.

(Gluing) Suppose $U$ is an open set, $\{U_{i}\}_{i\in I}$ is an open cover of $U$, and $\{s_{i}\in \mathcal{F}(U_{i})\}_{i\in I}$ is a family of sections. If all pairs of sections agree on the overlap of their domains, that is, if $s_{i}|_{U_{i}\cap U_{j}}=s_{j}|_{U_{i}\cap U_{j}}$ for all $ i,j\in I$, then there exists a section $ s\in \mathcal{F}(U)$ such that $ s|_{U_{i}}=s_{i}$ for all $i\in I.$

Presheaves $\mathcal{L}^{p,q}, \mathcal{D}^{p,0}$ satisfy Locality axiom for any open cover and Gluing axiom only for finite covers.
Indeed, if $\{U_{i}\}_{i\in I}$ is an infinite cover of $U,$ then it may happen that the norms of restriction to each
$U_{i}$ are finite but the norm of the element on $U$ is infinite, therefore it does not belong to $\mathcal{L}^{p,q}(U).$

It is possible to glue sections of $\mathcal{D}^{p,0},$ because there is partition of unity in the space of regular tropical $(0,0)-$forms.  Using that partition of unity one can check that the equation (\ref{eq.int_by_parts}) holds for the glued section.  

Indeed, suppose $U$ is an open set, $\{U_{i}\}_{i\in I}$ is an open finite cover of $U$, and $\{\psi_{i}\in \mathcal{D}^{p,0}(U_{i})\}_{i\in I}$ is a family of sections such that $\psi_{i}|_{U_{i}\cap U_{j}}=\psi_{j}|_{U_{i}\cap U_{j}}$ for all $ i,j\in I$. Then there is a section $\psi\in \mathcal{L}^{p,0}(U)$ such that $\psi|_{U_{j}}=\psi_{j}.$
Since forms $\omega_i=d'' \psi_{j}\in \mathcal{L}^{p,1}(U_i)$ agree on the overlaps of their domains, we can glue them to the global form $\omega\in \mathcal{L}^{p,1}(U).$ Let us show that  $\psi$ is an element of $\mathcal{D}^{p,0}(U),$
that is, for any $\varphi\in \mathcal{E}^{1-p,0}(\Gamma)$ with a compact support on $U$ holds:
$$ \int_U\omega\wedge \varphi = (-1)^{p+1}\int_U \psi\wedge d'' \varphi.$$
Let $\rho_i\in \mathcal{E}^{0,0}(\Gamma) ,i\in I$ be a partition of unity subordinate to the open cover $\{U_{i}\}_{i\in I}.$
Then $$\int_U\omega\wedge \varphi = \sum_{i\in I} \int_{U_i} \omega\wedge \rho_i \varphi=$$
 because $\rho_i \varphi$ has a compact support on $U_i,$ we get
$$=(-1)^{p+1}  \sum_{i\in I}  \int_{U_i} \psi\wedge d''( \rho_i \varphi) = (-1)^{p+1} \int_U \psi\wedge d'' \varphi.$$
\end{remark}

\subsection{The main technical lemma.}
\begin{lemma}\label{lm.main}
Let $U\cong [-\infty,a)$ be an open neighborhood of a degree $1$ vertex of $\Gamma,$ the vertex is identified with the point $-\infty.$ Given a form $\omega \in\mathcal{L}^{p,1}(U).$ 
\begin{enumerate}
 
 \item
If $p=0$ and $\omega=\omega(x) d''x,$ let us define a $(0,0)-$form $\psi=\psi(x)$, i.e., a function as follows:
$$\psi(x)= - \int^a_x \omega(t) dt.$$
The function $\psi$ is well-defined and belongs to $\mathcal{L}^{0,0}(U),$ and the following estimates holds:
$$|\psi(x)|\leq \sqrt{a-x} ||\omega(x) d''x||.$$
 \item
If $p=1$ and  $\omega=\omega(x) d'x \wedge d''x,$ let us define a $(1,0)-$form $\psi=\psi(x) d'x$ as follows:
$$\psi(x)= - \int^x_{-\infty} \omega(t) dt.$$
The form $\psi$ is well-defined and belongs to $\mathcal{L}^{0,1}(U),$ and the following estimates holds:
$$|\psi(x)|\leq\sqrt{\int^x_{-\infty} g(t) d t} ||\omega(x) d'x \wedge d''x||.$$
\item
The map $\omega \rightarrow \psi$ is a bounded linear operator from $\mathcal{L}^{p,1}(U)$ to $\mathcal{D}^{p,0}(U)$ and
 $d'' \psi = \omega.$
Let us denote this operator $T_U.$
\item
Suppose there is a form $\widetilde{\psi}\in\mathcal{D}^{p,1}(U)$ such that $d'' \widetilde{\psi} = \omega.$ 
Then, if $p=1,$  $\psi = \widetilde{\psi},$ and, if $p=0,$ $\psi = C + \widetilde{\psi}, C\in\mathbb{R}.$
\end{enumerate}
\end{lemma}

\begin{remark}
In other words, this lemma says that starting from a form $\omega\in \mathcal{L}^{p,1}(U)$ we can find a form 
$\psi \in \mathcal{L}^{p,1}(U)$ such that $d'' \psi = \omega.$
The value and the norm of $\psi$ can be estimated using the norm of $\omega.$ Also, the form $\psi$ is a unique form such that $d'' \psi = \omega$ if $p=1,$ 
or unique up to addition of a constant if $p=0.$
\end{remark}

\begin{proof}
Let us consider the case of $(1,1)-$form. 
Let $I_{\leq x}$ be the indicator function of the set $[-\infty,x].$
Given a form $\omega = \omega(x) d'x\wedge d''x\in \mathcal{L}^{1,1}(U).$
Then the $(0,1)-$form $\psi=\psi(x) d'x$ is defined as follows:  
$$\psi(x)=- \int_U I_{\leq x}(t) \omega(t) dt.$$ 
This integral is well-defined, indeed, it can be written using the scalar product on the space 
$\mathcal{L}^{1,1}(U):$ 
$$ \int_U I_{\leq x}(t) \omega(t) dt = - (\omega, I_{\leq x} g),$$ 
and the form $I_{\leq x} g$ is an element of  $\mathcal{L}^{1,1}(U).$

We have to show that $\psi \in \mathcal{L}^{1,0}(U).$
In particular, that $$||\psi||^2=\int_{U} \psi^2(x) dx < +\infty.$$
Recall that $$||\omega||^2=\int_U \frac{1}{g(x)} \omega^2(x) dx$$
Using the Cauchy-Schwarz inequality we get
$$|\psi(x)|=\bigg|\int_U I_{\leq x}(t) \omega(t) d t\bigg|=\bigg|\int_U \sqrt{g(t)} I_{\leq x}(t)  \frac{1}{\sqrt{g(t)}} \omega(t) \psi\bigg|\leq ||\omega|| \sqrt{\int_U g(t) I_{\leq x}(t) dt }  .$$
So we obtained the estimation:
$$|\psi(x)|\leq ||\omega||  \sqrt{\int^x_{-\infty} g(t) dt }.$$
Then
$$||\psi||^2=\int^a_{-\infty} \psi^2(x) dx \leq 
||\omega||^2 \int^a_{-\infty} dx (\int^x_{-\infty} g(t) dt)=$$
changing the order of integration we get
$$=||\omega||^2  \int^a_{-\infty} g(t) dt (\int^a_{t} dx )=||\omega||^2 \int^a_{-\infty} (a-t) g(t) dt$$
By the definition of the K\"{a}hler metric (Definition \ref{def.kh}) the integral $\int^a_{-\infty} (a-t) g(t) dt$ converges. 
Thus $$||\psi||\leq C ||\omega||$$
and the constant $C$ does not depend on the choice of $\omega.$ So there is a bounded linear operator $$T_U:\mathcal{L}^{1,1}(U) \rightarrow \mathcal{D}^{1,0}(U)$$
such that 
$T_U \omega = \psi.$

Let us check that $d''\psi$ is equal to $\omega.$
It means that for any for regular $(0,0)-$form $\varphi$ with the compact support in $U$ holds:
$$\int_U\omega\wedge \varphi = \int_U \psi\wedge d'' \varphi.$$
Consider the right hand side of the equality
$$\int_U \psi\wedge d'' \varphi = \int_U \psi(x) \varphi'(x) dx = \int^a_{-\infty} (-\int^x_{-\infty}\omega(t) dt) \varphi'(x) dx =$$
changing the order of integration we get
$$=\int^a_{-\infty} (-\int^a_{t}\varphi'(x) d x) \omega(t) dt=$$
by the Newton-Leibniz formula, since $\varphi(x)$ is equal to zero in a neighborhood of $a,$ we obtain
$$=\int^a_{-\infty} \varphi(t) \omega(t) dt = \int_U\omega\wedge \varphi.$$
Thus, $\psi$ belongs to $\mathcal{D}^{1,0}(U)$ and $d''\psi = \omega.$

Let us consider the case of $(0,1)-$forms. This case is quite similar to the previous one, but there are some minor differences.
Given a form $$\omega = \omega(x)  d''x\in \mathcal{L}^{0,1}(U).$$
Let $I_{ x\leq}$ be the indicator function of the set $[x,a).$
Consider the function $$\psi(x)=-\int_U I_{x\leq}(t) \omega(t) dt.$$
This integral is well-define since it is equal to the scalar product of two elements in  $\mathcal{L}^{0,1}(U):$
$$\psi(x)=-(\omega,I_{x\leq}(t) d''x).$$

We have to show that $\psi \in \mathcal{L}^{0,0}(U),$ thus we need to check that
$$||\psi||^2=\int_{U} g(x) \psi^2(x) dx < +\infty.$$
Since $$||\omega||^2=\int_U  \omega^2(x) dx$$
using the Cauchy-Schwarz inequality we get
$$|\psi(x)|=\bigg| \int_U I_{x\leq}(t) \omega(t) dt \bigg|\leq ||\omega|| \sqrt{\int_U I_{x\leq}(t) dt} = ||\omega|| \sqrt{a-x}.$$
Then
$$||\psi||^2=\int_{U} g(x) \psi^2(x) dx \leq ||\omega|| \int^a_{-\infty} (a-x) g(x) dx$$ 
By the definition of K\"{a}hler metric (Definition \ref{def.kh}) the integral $\int^a_{-\infty} (a-x) g(x) dx$ converges.
Thus $$||\psi||\leq C ||\omega||$$
and the constant does not depends on the choice of $\omega.$ So there is a bounded linear operator $$T_U:\mathcal{L}^{0,1}(U) \rightarrow \mathcal{D}^{0,0}(U)$$
such that 
$T_U \omega = \psi.$

Let us check that $d''\psi$ is equal to $\omega.$
It means that for any regular $(1,0)-$form $\varphi= \varphi(x) d'x$ with the compact support in $U$ holds:
$$\int_U\omega\wedge \varphi = - \int_U \psi\wedge d'' \varphi.$$
Consider the right hand side of the equality: 
$$-\int_U \psi\wedge d'' \varphi = -\int^a_{-\infty} (- \int^a_x  \omega(t) d t) (- \varphi'(x)) d x = $$
changing the order of integration we obtain
$$ = - \int^a_{-\infty} ( \int^t_{-\infty} \varphi'(x) d x ) \omega(t)  d t =$$
Since $\varphi$ is regular it equal to zero in a neighborhood of $-\infty.$ Thus, applying the Newton-Leibniz formula we obtain
$$=- \int^a_{-\infty} \varphi(t)  \omega(t)  d t= \int_U\omega\wedge \varphi.$$
Thus, $\psi$ belongs to $\mathcal{D}^{0,0}(U).$

Suppose there is a form $\widetilde{\psi}\in\mathcal{D}^{p,0}(U)$ such that $d'' \widetilde{\psi} = \omega.$ 
Then $d''(\widetilde{\psi}-\psi)=0,$ hence the coefficient of $\widetilde{\psi}-\psi$ should be constant, i.e.,
$\widetilde{\psi}-\psi$ is equal, if $p=1,$ to $c d'' x,$ or, if $p=0,$ to $c$ where $c\in \mathbb{R}.$
In the case $p=1,$
$$||c d''x||^2=\int^a_{-\infty} c^2 d x = +\infty,$$
hence $c=0$ and $\widetilde{\psi}=\psi.$
In the case $p=0,$ the integral
$$||c||^2=\int^a_{-\infty}c^2 g(x) dx $$
converges by the definition of the  K\"{a}hler form (Definition \ref{def.kh}) and $c$ can be any real number.
 
\end{proof}

\subsection{Relation to the Sobolev space.}
\begin{lemma}\label{lm.H1}
Suppose $U$ is an open subsets of $\Gamma$ and it is isometric to an open interval of the finite length $U\cong (a,b).$
Given a form $\psi\in\mathcal{D}^{p,0}(U),$ in terms of coordinates it is equal either to $\psi = \psi(x) d'x$ or to 
$\psi =\psi(x).$ 
Then the coefficient $\psi(x)$ belongs to the Sobolev space $H^1(U)=H^1(a,b).$
\end{lemma}
\begin{proof}
Suppose  $\psi\in\mathcal{D}^{0,0}(U).$ Let $\omega=\omega(x) d''x = d'' \psi.$ Then the norms of these elements are equal: 
$$||\psi||^2=\int^b_a \psi^2(x) g(x) dx,$$
$$||\omega||^2= \int^b_a \omega^2(x) d x.$$
Since $g(x)$ is a nonegative continuous function on the closure of $U,$ there are constants $0<c,C$ such that 
$c<g(x)<C$ for any $x\in U.$ Therefore, the norm on $\mathcal{L}^{0,0}(U)$ is equivalent to the standard norm on 
$L^2(a,b):$ 
$$c \int^b_a \psi^2(x) dx <||\psi||^2< C\int^b_a \psi^2(x) dx.$$
The norm of the $(1,1)-$form $\omega\in\mathcal{L}^{0,1}(U)$ is equal to the standard norm on $L^2(a,b)$ of its coefficient the function $\omega(x)$.

Consider equation (\ref{eq.int_by_parts}):
 $$\int_U \omega\wedge \varphi = -\int_U \psi\wedge d'' \varphi.$$
Since $\psi\in\mathcal{D}^{0,0}(U)$, it holds for any regular from $\varphi= \varphi(x) d'x$ with a compact support on $U.$
Therefore, the coefficient $\varphi(x)\in C^\infty_0(U)$ is a smooth function with a compact support on the interval $U=(a,b).$
In the terms of coefficient the equation looks like:
$$ \int^b_a \omega(x) \varphi(x) = - \int^b_a \psi(x) \varphi'(x),$$
where $\psi(x),\omega(x)\in L^2(a,b),$ and $\varphi(x)\in C^\infty_0(a,b).$
It is exactly the definition of the Sobolev space, hence $\omega(x)$ is the weak derivative of $\psi(x),$ and the function $\psi(x)$ belongs to the Sobolev space $H^1(U).$
\end{proof}
\begin{cor}
Given a form $\psi\in\mathcal{D}^{p,0}(\Gamma)$ and a vertex $v$ of $\Gamma$ of degree $\geq 2.$ 
There are well-defined boundary values of the coefficients of $\psi$ at $v$ along the edges incident to $v.$
\end{cor}
\begin{proof}
Let $e$ be an edge of $\Gamma$ incident to $v.$ Let $U$ be a finite open interval such that $U\subset e$ and the closure of $U$ contains $v.$ Then by Lemma \ref{lm.H1} the coefficient of the restriction $\psi|_U$ is a function from the Sobolev space $H^1(U),$
and, therefore, has a well-defined trace at the boundary of $U,$ in particular, at $v.$
\end{proof}

\subsection{$d''$ is a closed.}
\begin{proposition}
The operator $d''$ is a closed operator.
\end{proposition}
\begin{proof}
 
An operator is closed if its graph is closed.
Suppose $\omega_n\in \mathcal{D}^{p,0}(d''),$ $\omega_n\rightarrow \omega$ in $\mathcal{L}^{p,0}(\Gamma)$ and 
$d'' \omega_n \rightarrow \psi$ in $\mathcal{L}^{p,1}(\Gamma).$ Then the relation 
$$\int_\Gamma d'' \omega_n\wedge \varphi = (-1)^{p+1}\int_\Gamma \omega_n\wedge d'' \varphi$$
holds for any $\varphi\in \mathcal{E}^{1-p,0}(\Gamma).$

Since $** = (-1)^{p+q}\Id$ on the space of $(p,q)-$forms
and $(\alpha,\beta)=\int_{\Gamma}\alpha \wedge *\beta,$
we get 
$$\int_\Gamma d'' \omega_n\wedge \varphi=(-1)^{p-1}\int_\Gamma d'' \omega_n\wedge **\varphi=(-1)^{p-1}(d'' \omega_n,*\varphi),$$
and 
$$ (-1)^{p+1}\int_\Gamma \omega_n\wedge d'' \varphi=  - \int_\Gamma \omega_n\wedge ** d'' \varphi =-(\omega_n,* d'' \varphi).$$
Thus $$(-1)^{p}(d'' \omega_n,*\varphi)=(\omega_n,* d'' \varphi).$$
Taking limit as $n\rightarrow \infty$ we get
$$(-1)^{p}(\psi,*\varphi)=(\omega,* d'' \varphi).$$
This equation is equivalent to
$$\int_\Gamma \psi\wedge \varphi = (-1)^{p+1}\int_\Gamma \omega\wedge d'' \varphi.$$
Hence, $\psi$ is the weak $d''-$differential of $\omega.$
\end{proof}

\subsection{$L^2$-cohomology and the \v{C}ech to de Rham isomorphism.}
Let $H^{q}(\mathcal{L}^{p,*}(\Gamma),d'')$ be the cohomology group of the complex
$$0\rightarrow \mathcal{D}^{p,0}(\Gamma) \rightarrow \mathcal{L}^{p,1}(\Gamma) \rightarrow 0,$$
where $\mathcal{D}^{p,0}(\Gamma)\subset \mathcal{L}^{p,0}(\Gamma)$ is the domain of the operator $d''.$
\begin{theorem}\label{th.exct}
For a sufficiently small neighborhood $U$ of a point $x\in \Gamma$ there are exact sequences:   
\begin{equation}\label{eq.exct_shfs}
\begin{aligned} 
0\rightarrow\mathbb{R}_{\Gamma}(U)\xrightarrow{i}\mathcal{D}^{0,0}(U) \xrightarrow{d''} \mathcal{L}^{0,1}(U) \rightarrow 0,\\
0\rightarrow \Lambda^1_{\Gamma}(U)\xrightarrow{i}\mathcal{D}^{1,0}(U) \xrightarrow{d''} \mathcal{L}^{1,1}(U) \rightarrow 0,
\end{aligned}
\end{equation}
where $i$ is a natural inclusion of subpresheaves.
\end{theorem}
\begin{proof}
Firstly, we will prove the following statement.
\begin{lemma}\label{lm.invers}
 Let $U$ be a sufficiently small neighborhood of a point in $\Gamma$ then there is a bounded operator 
 $$T_U:\mathcal{L}^{p,1}_{\Gamma}(U)\rightarrow \mathcal{D}^{p,0}_{\Gamma}(U)$$
 such that $d'' T_U = \Id.$ 
\end{lemma}
\begin{proof}
To prove the statement we will consider several distinct cases: the neighborhood $U$ can be a neighborhood of a degree $1$ vertex, or of a degree $n\geq 2$ vertex, or of an internal point of an edge; $p$ can be equal to $0$ or $1.$

Let $U\simeq[-\infty,a)$ be a neighborhood of a degree $1$ vertex. We identify $-\infty$ with this vertex.
In this case the required operator $ T_U:\mathcal{L}^{p,1}_{\Gamma}(U)\rightarrow \mathcal{D}^{p,0}_{\Gamma}(U)$ was constructed in Lemma \ref{lm.main}.

Given a vertex $v$ of $\Gamma$ of degree $d\geq2.$ Consider a neighborhood $U$ of $v:$   
$$ U =\bigsqcup_{\mbox{$d$-times}} (-a,0]/\sim,$$
where points $0$ of the different intervals are all identified by the equivalence relation $\sim.$ The class of $0$ is identified with the vertex $v$. 

Given a form $\omega\in\mathcal{L}^{1,1}(U).$ Suppose $e_i\simeq(-a,0]\subset U, i=1,\dots,d$ is the $i-$th edge of $U$ 
and $\omega_i=\omega_i(x)  d'x \wedge d''x$ is the restriction of $\omega$ to $e_i.$
Let us define $\psi= T_U \omega,$ where $\psi_i=\psi_i(x) d' x$ is the restriction of $\psi$ to $e_i,$ as follows:
$$\psi_i(x)=\int^0_x  \omega_i(t) dt.$$
These functions $\psi_i(x)$ are well-defined since they can be expressed as a scalar product $$\psi_i(x)=(\omega,I_{i,x} g),$$ where $\omega,I_{i,x} g \in \mathcal{L}^{1,1}(U),$ and $I_{i,x}$ is the indicator function of the set $[x,0]\subset e_i.$

Since $\psi_i(0)=0,i=1,\dots,d,$ the Kirchhoff's law holds for $\psi.$
Using the same arguments as in Lemma \ref{lm.main} we can check that 
$\psi \in \mathcal{D}^{1,0}(U),$ $d'' \psi = \omega,$ and $T_U$ is bounded.  

Finally, if $\omega\in\mathcal{L}^{0,1}(U),$ and $\omega_i= \omega_i(x) d''x$ is the restriction of $\omega$ to $e_i.$ 
Let us define $\psi= T_U \omega,$ where $\psi_i=\psi_i(x)$ is the restriction of $\psi$ to $e_i,$ as follows:  
$$\psi_i(x)=-\int^0_x  \omega_i(t) dt.$$
Since  $\psi_i(0)=0,$ the Continuity property hold at the vertex $v.$ Again, using the same arguments as above we can check that $\psi \in \mathcal{D}^{0,0}(U),$ $d'' \psi = \omega,$ and $T_U$ is bounded.  

The case of a neighborhood of an internal point of an edge is equivalence to the case of a neighborhood of degree $2$ vertex.
\end{proof}

Let $U$ be a sufficiently small neighborhood of a point in $\Gamma.$
The kernel of $d'':\mathcal{D}^{p,0}_{\Gamma}(U) \rightarrow \mathcal{L}^{p,1}_{\Gamma}(U)$ coincides with
$\mathbb{R}_{\Gamma}(U)$ or $\Lambda^1_{\Gamma}(U).$
By Lemma \ref{lm.invers} the map $d'':\mathcal{D}^{0,0}_{\Gamma}(U) \rightarrow \mathcal{L}^{0,1}_{\Gamma}(U)$ is surjective.
Therefore, the sequences (\ref{eq.exct_shfs}) are exact.
\end{proof}

\begin{proposition}There is an isomorphism
$$H^{p,q}(\Gamma)\cong H^{q}(\mathcal{L}^{p,*}(\Gamma),d'').$$
\end{proposition}
\begin{proof}
Using the exact sequences (\ref{eq.exct_shfs}) we can repeat the proof of the Proposition \ref{pr.smooth_CdR}.
\end{proof}

\subsection{Integration by parts for weakly $d''-$differentiable froms.}
\begin{proposition}\label{pr.intbypart}
If $\psi\in\mathcal{D}^{0,0}(\Gamma)$ and $\varphi\in\mathcal{D}^{1,0}(\Gamma)$ are two  weakly $d''-$differentiable froms,
then the equation of integration by parts holds: 
\begin{equation}\label{eq.intbypart}
 \int_\Gamma d'' \psi\wedge \varphi = - \int_\Gamma \psi\wedge d'' \varphi.
\end{equation}
\end{proposition}
 
\begin{proof}
 
Firstly, let us notice that both integrals in (\ref{eq.intbypart}) are well-defined, i.e., convergent. 
Indeed, consider the integral $\int_\Gamma d'' \psi\wedge \varphi,$ using the property of the Hodge star $**=\pm\Id$
we can rewrite it as
$- \int_\Gamma d'' \psi\wedge **\varphi.$ Thus, it is equal to $-(d'' \psi,*\varphi).$
Since $d'' \psi\in \mathcal{L}^{0,1}(\Gamma)$ and the Hodge star is, in this case, is an isomorphisms between $\mathcal{L}^{1,0}(\Gamma)$ and $\mathcal{L}^{0,1}(\Gamma),$ this scalar product is well-defined and, consequently, the integral is well-defined. We can apply the same argument for the second integral.

Let us choose a function $\rho_0\in \mathcal{E}^{0,0}(\Gamma)$ such that 
\begin{enumerate}
 \item its values between $0$ and $1;$
 \item it is equal to $1$ on each finite-length edge and in a neighborhood of any vertex of degree $\geq 2;$
 \item it is equal to $0$ on a neighborhood of any vertex of degree $1,$ i.e., in neighborhoods of infinite tails of the toropical curve.
\end{enumerate}
Let us denote $\rho_1=1-\rho_0.$
These two function, $\rho_0,\rho_1$ give us a partition of unity such that one of them is nonzero on a finite part of the curve another on the infinite tails.

Consider the integral
$$\int_\Gamma \psi\wedge d'' \varphi= \int_\Gamma \rho_0 \psi\wedge d'' \varphi + \int_\Gamma \rho_1 \psi\wedge d'' \varphi.$$

The support of $\rho_0$ is a union of finite-length edges and compact parts of infinite-length edges.
Using Lemma \ref{lm.H1} we obtain that $\psi$ and $\varphi$ has $H^1-$coefficients in a neighborhood of $\mathrm{supp}\;\rho_0.$
For the $H^1-$function we can apply integration by parts and the boundary terms at vertices vanish by the same reasons as in Stokes' Theorem (Theorem \ref{th.stokes}).

The second integral $\int_\Gamma \rho_1 \psi\wedge d'' \varphi$ is a sum of integrals over infinite-length edges. Suppose an infinite-length edge $e$ is isomorphic to $e\cong[-\infty,0],$ the function $\rho_1$ is equal to $0$ at a neighborhood of the point $0$ and equal to $1$ at a neighborhood of $-\infty.$ From this moment let write $\overline{\psi}$
instead of $\rho_1 \psi.$
We are going to prove that 
$$\int_{[-\infty,0]} d'' \overline{\psi} \wedge \varphi = - \int_{[-\infty,0]} \overline{\psi}\wedge d'' \varphi.$$
Then we can take a sum over all infinite-length edges this will prove the statement of the proposition.

The tropical integrals by definition equals: $$\int_e \overline{\psi}\wedge d'' \varphi = - \int^0_{-\infty} \overline{\psi}(t) \varphi'(t) d t$$ and
 $$\int_e d'' \overline{\psi}\wedge \varphi = - \int^0_{-\infty}  \overline{\psi}'(t) \varphi(t) dt.$$
 Therefore, we have to show that 
 \begin{equation}\label{eq.intbypart2}
  \int^0_{-\infty} \overline{\psi}(t) \varphi'(t) d t = - \int^0_{-\infty}  \overline{\psi}'(t) \varphi(t) dt.
 \end{equation}
Let us split both parts of the equality to sums of integrals: 
$$\int^0_{x} \overline{\psi}(t) \varphi'(t) d t + \int^x_{-\infty} \overline{\psi}(t) \varphi'(t) d t = - \int^0_{x}  \overline{\psi}'(t) \varphi(t) dt - \int^x_{-\infty}  \overline{\psi}'(t) \varphi(t) dt,$$
where $x\in (-\infty,0].$
Since our initial integrals are convergent, we have  $$\lim_{x\rightarrow -\infty} \int^x_{-\infty} \overline{\psi}(t) \varphi'(t) d t = 0$$ and 
$$\lim_{x\rightarrow -\infty} \int^x_{-\infty}  \overline{\psi}'(t) \varphi(t) dt =0.$$

By Lemma \ref{lm.H1} the restrictions of $\overline{\psi}(x),\varphi(x)$ to any interval $(x,0)\subset [-\infty,0] \cong e, x\in\mathbb{R}$ are functions form the Sobolev space $H^1(x,0).$
Thus, we can apply integration by parts on $(x,0)$:
$$\int^0_{x} \overline{\psi}(t) \varphi'(t) d t = \overline{\psi}(t) \varphi(t)|^{0}_{x} - \int^0_{x}  \overline{\psi}'(t) \varphi(t) dt.$$

Consider the term $\overline{\psi}(t) \varphi(t)|^{0}_{x} = \overline{\psi}(0) \varphi(0) - \overline{\psi}(x) \varphi(x).$
The first summand $\overline{\psi}(0)= \rho_1(0) \psi(0) =0$ is equal to zero. We are going to show that 
$$\lim_{x\rightarrow -\infty} \overline{\psi}(x) \varphi(x) = 0.$$

Consider the forms $\omega = d'' \varphi$ and $\tau=d'' \overline{\psi}.$
In the local coordinates we have:
$$\overline{\psi}=\overline{\psi}(x),\;\varphi=\varphi(x) d'x,\;\omega = \omega(x) d'x\wedge d''x,\;\tau=\tau(x)d''x.$$

Suppose $U=[-\infty,0)\subset e.$ The operator $$T_U:\mathcal{L}^{p,1}_{\Gamma}(U)\rightarrow \mathcal{D}^{p,0}_{\Gamma}(U)$$ was defined in Lemma \ref{lm.invers}. It has the following properties $d'' T_U = \Id,$ $ T_U d'' \varphi = T_U\omega = \varphi,$
and $T_U d'' \overline{\psi} = T_U \tau = \overline{\psi}+ c$ where $c$ is some constant. By the definition of $T_U$ we get:   
$$ \overline{\psi}(x)+ c = T_U \tau = -  \int^0_{x} \tau(t) dt,$$
and $$\varphi(x) d'x = T_U \omega= - (\int^x_{-\infty} \omega(t) dt) d'x.$$
Since $\overline{\psi}(0)= \rho_1(0) \psi(0)=0$ and $\psi(0)+c=\int^0_0 \tau(t) dt,$
the constant $c$ is equal to $0,$ and $\overline{\psi}= T_U \tau.$

The the following estimates was proven for $T_U:$
$$|\varphi(x)| \leq ||\omega|| \sqrt{\int^x_{-\infty} g(t) dt},$$
$$|\overline{\psi}(x)|\leq ||\tau|| \sqrt{|x|}$$

By definition of K\"{a}hler metric (Definition \ref{def.kh}) the integral $\int^0_{-\infty} t^2 g(t) dt$ converges 
which is equivalent to
$$\lim_{x\rightarrow -\infty} \int^x_{-\infty} t^2 g(t) dt =0.$$
For  $t\leq x<0$  we have $ x^2 g(t)  \leq t^2 g(t)$ and 
$$ x^2 \int^x_{-\infty} g(t) dt \leq \int^x_{-\infty} t^2 g(t) dt$$
which is equivalent to
$$ \int^x_{-\infty} g(t) dt \leq \frac{1}{x^2} \int^x_{-\infty} t^2 g(t) dt.$$

Finally, combining all our estimates we get 
$$|\varphi(x) \overline{\psi}(x)|\leq ||\omega|| \cdot ||\tau|| \sqrt{|x|} \sqrt{\int^x_{-\infty} g(t) dt} 
\leq ||\omega|| \cdot ||\tau|| \frac{1}{\sqrt{|x|}} \int^x_{-\infty} t^2 g(t) dt.$$
The right hand side tends to $0$ as $x$ tends to $-\infty.$
\end{proof}

\subsection{The adjoint of $d''$.}
Let us consider the adjoint operator 
$$d''^*:\mathcal{L}^{p,1}(\Gamma) \rightarrow \mathcal{L}^{p,0}(\Gamma).$$ 
By definition, $d''^* \omega = \psi$ if for any $\varphi\in \mathcal{D}^{p,0}(\Gamma)$ holds
$$(d'' \varphi, \omega)=(\varphi,\psi).$$
\begin{proposition}\label{pr.adj}
 
 The adjoint operator $$d''^*: \mathcal{L}^{p,1}(\Gamma) \rightarrow \mathcal{L}^{p,0}(\Gamma)$$ 
 is densely defined and closed. It is equal to $$d''^* = -*d*.$$
 In particular, $\psi\in \mathcal{D}(d''^*)$ if and only if $$* \psi \in \mathcal{D}^{1-p,0}(\Gamma).$$ Its adjoint  $d''^{**}$ equals $d''.$
\end{proposition}
\begin{proof}
From the general properties of unbounded operators follows that the adjoin $A^*$ of a closed densely defined operator $A$ 
is a closed densely defined operator and its adjoint  $A^{**}$ equals $A,$ \cite[Theorem 3.1 and Theorem 3.3]{BSU}.  
Thus we have to prove that $d''^*=-*d*.$

If $A: H_1 \rightarrow H_2$ is an unbounded operator with a domain $\mathcal{D}(A)\subset H_1$ then its adjoint $A^*$
is an unbounded operator $A^*: H_1 \rightarrow H_2$ such that for any $x\in \mathcal{D}(A)\subset H_1$ holds 
$$(x, A^*y)_{H_1}=(A x, y)_{H_2},$$
the domain of $A^*$ is a maximal subspace of elements in $H_2$ satisfying that relation, i.e.,
$$\mathcal{D}(A^*)=\{x\in H_2: \exists z\in H_1, \forall y \in \mathcal{D}(A): (y,z)_{H_1}=(A y, x)_{H_2}  \}.$$

Suppose a form $\psi \in \mathcal{L}^{p,1}(\Gamma)$ is in the domain of $d''^*,$ $\psi\in \mathcal{D}(d''^*)$ and $\omega= d''^* \psi.$
Then by the definition of the adjoint operator for any form $\varphi\in\mathcal{D}^{p,0}(\Gamma)$ the following equality holds: 
$$(\varphi,\omega)=(d'' \varphi,\psi).$$
We can rewrite it as follows:
$$\int_{\Gamma}\varphi\wedge \ast \omega = \int_{\Gamma}d'' \varphi \wedge \ast \psi.$$
Since $\mathcal{E}^{p,0}(\Gamma)$ is a subspace of $\mathcal{D}^{p,0}(\Gamma)$  this equality hold for any regular form $\varphi \in \mathcal{E}^{p,0}(\Gamma).$
Hence, by the definition, $*\psi$ is $d''-$weakly differentiable form and its differential is equal to 
$$d''* \psi= (-1)^{p+1}*\omega.$$ Applying the Hodge star operator to both parts of the previous equality we get
$$*d''* \psi= (-1)^{p+1}**\omega = - \omega.$$
Hence $$d''^* \psi = \omega = - *d''* \psi.$$

At this moment we proved that if $\psi\in \mathcal{D}(d''^*)$ then $*\psi$ is  $d''-$weakly differentiable.
Let us prove the converse: if $*\psi$ is  $d''-$weakly differentiable then $\psi\in \mathcal{D}(d''^*).$
By Proposition \ref{pr.intbypart}, we can integrate by parts a product of two $d''-$weakly differentiable forms, i.e.,  if form  $*\psi$ is $d''-$weakly differentiable, then for any form $\varphi\in \mathcal{D}^{p,0}(\Gamma)$ holds:
$$\int_ \Gamma d'' \varphi\wedge  *\psi= (-1)^{p+1} \int_\Gamma \varphi\wedge d'' *\psi =   \int_\Gamma \varphi\wedge * (-* d'' *\psi).$$
That equality can be written as
$$(d'' \varphi,\psi)=(\varphi,-* d'' *\psi).$$

Thus we proved that the domain of  $\mathcal{D}(d''^*)$ coincides with the space of forms such that theirs Hodge stars are $d''-$weakly differentiable.
\end{proof}

\subsection{The Laplace-Beltrami operator and harmonic tropical superforms.}
Let us define the \emph{Laplace-Beltrami operator} as follows 
$$\Delta=d'' d''^* + d''^* d'': \mathcal{L}^{p,q}(\Gamma) \rightarrow \mathcal{L}^{p,q}(\Gamma).$$
It's domain equals 
$$\mathcal{D}(\Delta)=\{\omega\in \mathcal{L}^{p,q}(\Gamma): \omega\in \mathcal{D}(d''^*), \omega\in \mathcal{D}(d''),  (d''\omega)\in \mathcal{D}(d''^*), (d''^*\omega)\in \mathcal{D}(d'')\}.$$
By the dimensional reasons one of the summands in $\Delta$ is identically equal to zero, so $\Delta$ is either equals  $\Delta=d'' d''^*$ or  $\Delta=d''^* d''.$ 

\begin{proposition}\label{pr.har_cl}
Let $\omega$ be an element of $\mathcal{L}^{p,q}(\Gamma),$ then
$\Delta \omega=0$ if and only if $d''\omega=0$ and $d''^*\omega =0.$
\end{proposition}
\begin{proof}
Suppose $\Delta \omega=0.$ If $\omega\in \mathcal{D}(\Delta),$ then $\omega\in \mathcal{D}(d''^*)\cap \mathcal{D}(d'').$ 
By Proposition \ref{pr.adj} we have $d''^{**}=d''.$  From the definition of an  adjoint operator we get
$$0=(\Delta \omega,\omega)=(d'' d''^* + d''^* d'' \omega,\omega)=(d''^* \omega,d''^*\omega)+(d'' \omega,d''^{**}\omega)=(d''^* \omega,d''^*\omega)+(d'' \omega,d''\omega).$$ Hence $||d''^* \omega||=0, ||d'' \omega||=0,$ and $d''^* \omega = d''\omega =0.$ The converse follows directly from the definition of the Laplace-Beltrami operator.
\end{proof}

\begin{definition}
Let us denote the kernel of the Laplace-Beltrami operator $\Delta:\mathcal{L}^{p,q}(\Gamma)\rightarrow \mathcal{L}^{p,q}(\Gamma)$ by
$\mathcal{H}^{p,q}(\Gamma).$ We call this space $\mathcal{H}^{p,q}(\Gamma)$ the space of \emph{harmonic tropical superform} of degree $(p,q)$ on $\Gamma.$ 
\end{definition}

By Proposition \ref{pr.har_cl} any harmonic superform is closed, hence there is the map
$i:\mathcal{H}^{p,q}(\Gamma) \rightarrow H^{p,q}(\Gamma),$ that maps any harmonic form to its class in the cohomology group.

\begin{proposition}\label{pr.hode_iso} 
  The map $i:\mathcal{H}^{p,q}(\Gamma) \rightarrow H^{p,q}(\Gamma)$ is an isomorphism.
\end{proposition}
\begin{proof}
Let $\omega$ be an element of $\mathcal{D}^{p,0}(\Gamma).$
By Proposition $\ref{pr.har_cl},$ $d''\omega=0$  if and only if $\Delta \omega=0.$
Thus, $ H^{p,0}(\Gamma)=\ker d'' = \mathcal{H}^{p,0}(\Gamma).$

\begin{lemma}
The range  $\im \, d''$ of $d''$ is closed. 
\end{lemma}
\begin{proof}
Let $\mathfrak{U}=\{U_i\}_i$ be a cover of $\Gamma.$ 
Let us denote $U_{ij}=U_i\cap U_j.$

Since $\Gamma$ is compact, we may choose $\mathfrak{U}$ in such a way that: 
\begin{enumerate}
 \item $\mathfrak{U}$ is finite cover;
 \item the sequences of sections (\ref{eq.exct_shfs}) are exact over any $U_i$ and $U_{ij}=U_i\cap U_j;$
 \item there is bounded operator $T_U$ as in Lemma \ref{lm.invers} for any $U_i$ and $U_{ij}=U_i\cap U_j.$
\end{enumerate}

Let $C_i(\mathcal{S})$ be the \v{C}ech complex of a presheaf $\mathcal{S}$ and the cover $\mathfrak{U}$ with the differential  $\delta.$ 
In particular, $$C_0(\mathcal{S})=\bigoplus_i\mathcal{S}(U_i),\; C_1=\bigoplus_{i<j}\mathcal{S}(U_{ij}),$$
 and $\delta:C_0(\mathcal{S}) \rightarrow C_1(\mathcal{S}).$

Since $C_i(\mathcal{L}^{p,q})$ is a direct sum of $\mathcal{L}^{p,q}(U)$ it has a structure of a Hilbert space induced from the summands. 
Then $\delta$ is a continuous linear operator.
The kernel $\mathrm{ker}\delta: C_0(\mathcal{L}^{p,q}) \rightarrow C_1(\mathcal{L}^{p,q})$ coincides with $\mathcal{L}^{p,q}(\Gamma).$
Actually, the norm on this kernel does not coincide with the norm on $\mathcal{L}^{p,q}(\Gamma),$ but these two norms are equivalent.
Since $\mathcal{L}^{p,q}(\Gamma)$ is the kernel of a bounded operator, it is a closed subspace of $C_0(\mathcal{L}^{p,q})$.
We will consider $\mathcal{L}^{p,q}(\Gamma)$ as a subspace of $C_0(\mathcal{L}^{p,q})$

The bounded operator $T_U: \mathcal{L}^{p,1}(U) \rightarrow \mathcal{L}^{p,0}(U)$ was defined in Lemma \ref{lm.invers}.
Let $T:C_0(\mathcal{L}^{p,1})\rightarrow C_0(\mathcal{L}^{p,0})$ be a direct sum of $T_{U_i}.$
The composition of operators $\delta T: \mathcal{L}^{p,1}(\Gamma) \rightarrow C_1(\mathcal{L}^{p,0})$ is a continuous linear operator.
Combining the facts that $d'' T\varphi= \varphi,$ the operator $\delta$ commutes with $d'',$ and $\mathrm{ker}\delta = \mathcal{L}^{p,q}(\Gamma),$ for any $\omega\in \mathcal{L}^{p,1}(\Gamma)$ we get $$d'' \delta T \omega =\delta d''  T \omega= \delta \omega = 0.$$  Therefore $\delta T$ is a bounded operator from $\mathcal{L}^{p,1}(\Gamma)$ to the kernel of 
$$d'':  C_1(\mathcal{L}^{p,0}) \rightarrow C_1(\mathcal{L}^{p,1}),$$ which is equal to either $C_1(\mathbb{R}_\Gamma)$ or 
$C_1(\Lambda^1_\Gamma).$  Both spaces $C_1(\mathbb{R}_\Gamma)$ and $C_1(\Lambda^1_\Gamma)$ are finite dimensional. 
There are the quotient maps $$\varepsilon:C_1(\mathbb{R}_\Gamma) \rightarrow H^1(\Gamma,\mathbb{R}_\Gamma)=C_1(\mathbb{R}_\Gamma)/\delta C_0(\mathbb{R}_\Gamma)$$ and $$\varepsilon:C_1(\Lambda^1_\Gamma) \rightarrow H^1(\Gamma,\Lambda^1_\Gamma)=C_1(\Lambda^1_\Gamma)/\delta C_0(\Lambda^1_\Gamma).$$ These maps are continuous because these are linear maps between finite-dimensional vector spaces. 

The kernel of $\varepsilon \delta T $ coincides with $\im \, d''$ in $\mathcal{L}^{p,1}(\Gamma)$.
Indeed, assume  $\omega\in \mathcal{L}^{0,1}(\Gamma)$ and $\varepsilon \delta T \omega = 0,$ then there is a cochain 
$\psi\in C_0(\mathbb{R}_\Gamma)$ such that $\delta T \omega = \delta\psi.$ Since $\delta ( T \omega - \psi) = 0,$ 
we get $T \omega - \psi \in \mathcal{L}^{0,0}(\Gamma)$ and $d'' (T \omega - \psi) = \omega.$ The same arguments works for 
the case  $\omega\in \mathcal{L}^{1,1}(\Gamma).$
Since $\varepsilon \delta T$ is continuous, the kernel is closed, and, consequently, $\im d''$ is closed.
\end{proof}

Since $\im \, d''$ is closed, there is the decomposition $$\mathcal{L}^{p,1}(\Gamma)=\im \, d'' \oplus (\im \, d'')^\perp.$$
Hence $$H^{1}(\mathcal{L}^{p,*}(\Gamma),d'') = \mathcal{L}^{p,1}(\Gamma)/\im\, d'' \cong (\im\, d'') ^\perp$$
The kernel of a closed densely defined operator coincides with the orthogonal complement of the range of the adjoint.
Thus by Proposition \ref{pr.adj} we get $(\im\, d'') ^\perp = \ker{d''^*}.$
By Proposition $\ref{pr.har_cl}$ an element $\omega \in \mathcal{L}^{p,1}(\Gamma)$ is harmonic if and only if $d''^* \omega=0,$
thus $\mathcal{H}^{p,1}(\Gamma)=\ker d''^*$ and $i:\mathcal{H}^{p,1}(\Gamma) \rightarrow H^{1}(\mathcal{L}^{p,*}(\Gamma),d'')$ 
is an isomorphism.\end{proof}

\begin{theorem}
 
 The Laplace-Beltrami operator is a self-adjoint operator.
\end{theorem}
\begin{proof}
By the dimensional reasons the operator $\Delta$ is equal to either $\Delta = d'' d''^*$ or $\Delta = d''^* d''.$
By von Neumann theorem for any closed densely operator $A$ the operator $A^* A$ is a self-adjoint operator \cite[Theorem 7.3]{BSU}.
Thus by Proposition \ref{pr.adj} both $d'' d''^*$ and $d''^* d''$ are self-adjoint.
\end{proof}

\begin{remark}
Let us describe the operator $\Delta$ in terms of local coordinates.
Let $x$ be a local coordinate on an edge of $\Gamma$ and the K\"{a}hler form $g$ is locally given by the equation $g=g(x) d' x \wedge d'' x.$
Then by straightforward computation we obtain
$$\Delta (f(x)) = - \frac{1}{g(x)}\frac{\partial^2 f(x)}{\partial x^2},$$
$$\Delta (f(x) d' x) = - \frac{1}{g(x)}\frac{\partial^2 f(x)}{\partial x^2} + \frac{1}{g^2(x)} \frac{\partial f(x)}{\partial x}  \frac{\partial g(x)}{\partial x} d'x ,$$
$$\Delta (f(x) d'' x) = - \frac{1}{g(x)}\frac{\partial^2 f(x)}{\partial x^2} + \frac{1}{g^2(x)} \frac{\partial f(x)}{\partial x}  \frac{\partial g(x)}{\partial x} d''x ,$$
$$\Delta (f(x) d' x\wedge d'' x) = - \frac{1}{g(x)}\frac{\partial^2 f(x)}{\partial x^2} +  \frac{2}{g^2(x)} \frac{\partial f(x)}{\partial x}  \frac{\partial g(x)}{\partial x} + \frac{2}{g^2(x)}  \frac{\partial^2 g(x)}{\partial x^2} f(x)  - 
\frac{2}{g^3(x)} (\frac{\partial g(x)}{\partial x})^2 f(x) d' x\wedge d''x .$$
\end{remark}

\begin{proposition}\label{pr.commute}
 The Hodge star operator commutes with the Laplace-Beltrami operator, i.e., 
 $* \Delta = \Delta *$ and 
 $\varphi\in\mathcal{D}(\Delta:\mathcal{L}^{p,q}(\Gamma)\rightarrow \mathcal{L}^{p,q}(\Gamma))$ if and only if $*\varphi\in \mathcal{D}(\Delta:\mathcal{L}^{1-p,1-q}(\Gamma)\rightarrow \mathcal{L}^{1-p,1-q}(\Gamma)).$
\end{proposition}
\begin{proof}
Firstly, let us check that  $\varphi\in\mathcal{D}(\Delta:\mathcal{L}^{p,q}(\Gamma)\rightarrow \mathcal{L}^{p,q}(\Gamma))$ if and only if $*\varphi\in \mathcal{D}(\Delta:\mathcal{L}^{1-p,1-q}(\Gamma)\rightarrow \mathcal{L}^{1-p,1-q}(\Gamma)).$

The domain of $\Delta$ equals either
$$\mathcal{D}(\Delta:\mathcal{L}^{p,0}(\Gamma)\rightarrow \mathcal{L}^{p,0}(\Gamma))=\{ \omega: \omega \in \mathcal{D}(d''), d''\omega \in \mathcal{D}(d''^*)\},$$
or
$$\mathcal{D}(\Delta:\mathcal{L}^{p,1}(\Gamma)\rightarrow \mathcal{L}^{p,1}(\Gamma))=\{ \omega: \omega \in \mathcal{D}(d''^*), d''^*\omega \in \mathcal{D}(d'')\}.$$
By Proposition \ref{pr.adj}
$\omega \in \mathcal{D}(d'')$ if and only if $*\omega \in \mathcal{D}(d''^*),$ and $d''^* = - * d''*.$
 Thus,  $d'' \omega\in \mathcal{D}(d''^*)$ if and only if $* d'' \omega \in \mathcal{D}(d'').$
Using the equality $$* d'' = \pm * d''* *=\pm d''^* *$$ and the previous statement we get 
 $d'' \omega\in \mathcal{D}(d''^*)$ if and only if $d''^* *\omega \in \mathcal{D}(d'').$
 Thus $\omega \in \mathcal{D}(d'')$  and  $d''\omega \in \mathcal{D}(d''^*)$ if and only if $*\omega \in \mathcal{D}(d''^*)$ and $d''^* *\omega \in \mathcal{D}(d'')$ which is equivalent to
 $$\omega\in \mathcal{D}(\Delta)  \Longleftrightarrow *\omega\in \mathcal{D}(\Delta).$$

Now, let us check the commutativity $* \Delta = \Delta *.$ The equality $* \Delta = \Delta *$
can be written as 
$$* \Delta = - * d'' *d''* -* *d''* d'' = - d'' *d''* *  - *d'' *d''* = \Delta *.$$
Since $**=(-1)^{p+q}\Id,$ as an operator on  $\mathcal{L}^{p,q}(\Gamma),$ we get 
$$- * d'' *d''* + (-1)^{p+q+1}d''* d'' = (-1)^{p+q+1} d'' *d''  - *d'' *d''*.$$
Thus the equality $* \Delta = \Delta *$ holds.
\end{proof}

\subsection{The main result and final remarks.}
The main result of this paper is the following
\begin{theorem}\label{th.hodge}
 Let $\Gamma$ be a tropical curve of genus $n.$
 The Hodge star operator maps harmonic superform to harmonic superform and the map
 $*:\mathcal{H}^{p,q}(\Gamma) \rightarrow \mathcal{H}^{1-p,1-q}(\Gamma)$ is an isomorphism, and, 
 consequently,  $H^{p,q}(\Gamma)\simeq H^{1-p,1-q}(\Gamma).$
 In particular, $$H^{1,1}(\Gamma)\simeq H^{0,0}(\Gamma) \simeq H^{0}(\Gamma,\mathbb{R})\cong \mathbb{R}$$ and
  $$H^{1,0}(\Gamma)\simeq H^{0,1}(\Gamma) \simeq H^{1}(\Gamma,\mathbb{R}) \cong \mathbb{R}^n.$$ 
\end{theorem}
\begin{proof}
By Proposition \ref{pr.commute} if $\omega$ is an element of $\mathcal{H}^{p,q}(\Gamma),$ then $*\omega$ is also a harmonic form, $*\omega\in \mathcal{H}^{1-p,1-q}(\Gamma)$.
Since $**=\pm\Id,$ we get that $*$ is an isomorphism between $\mathcal{H}^{p,q}(\Gamma)$ and $\mathcal{H}^{1-p,1-q}(\Gamma).$
Thus, by Proposition \ref{pr.hode_iso} we get $H^{p,q}(\Gamma)\simeq H^{1-p,1-q}(\Gamma).$
Since by definition $H^{0,q}(\Gamma)$ is the cohomology group of the sheaf $\mathbb{R}_\Gamma$ of locally constant functions,
$H^{0,q}(\Gamma)$ is isomorphic to the usual topological cohomology group $H^{q}(\Gamma,\mathbb{R}).$ 
\end{proof}

The space $H^{0,0}(\Gamma)$ is generated by a constant function. Since the Hodge star of a constant function is 
proportional to the K\"{a}hler form $g,$ the class of $g$ is a generator of $H^{1,1}(\Gamma).$
The space $\mathcal{H}^{1,0}(\Gamma)$ is the space of differential forms with coefficients constant on edges and satisfying the Kirchhoff's law and the Regularity at infinity conditions. Since $* d'x=d'' x,$ the group $\mathcal{H}^{0,1}(\Gamma) \cong H^{0,1}(\Gamma)$ is generated by essentially the same differential forms which 
are considered as $(0,1)-$forms. 
\begin{remark}
Theorem \ref{th.hodge} is a toropical analog of the Hodge theory on a compact Riemann surface.
We proved this theorem using methods of the Hodge theory, but one can prove that there is an isomorphism $H^{p,q}(\Gamma)\simeq H^{1-p,1-q}(\Gamma)$
using a quite simple combinatorial methods. So for the purpose of this result our paper is overcomplicated. One can consider this paper as a proof of concept for the Hodge theory on higher dimensional tropical varieties. 
\end{remark}

\begin{remark}
Now we would like to discuss the relation of our paper to the quantum graphs. 

The research in quantum graphs is mostly devoted to the study of the Schr\"{o}dinger equation on metric graphs. 
This study usually based on study of stationary states, i.e., eigenfunctions of the Laplace operator.
For the Laplace operator to be self-adjoint some boundary conditions at the vertexes of a graph are needed.
There are a variety of such boundary conditions, some of them resembles our boundary conditions.

There is a difference between out approach and the standard quantum graph theory.
Usually only functions are considered, but we also consider differential forms and tensor fields.
We use harmonic forms as a tool to compute some cohomologies and the Laplace-Beltrami operator arise from the chain complex. Usually quantum graphs are not related to the study of cohomologies and chain complexes. 
In the quantum graphs setting the whole spectrum of the Laplace operator is studied, but we are only working with harmonic functions and forms, i.e., with zero-eigenvectors.

In our case there is a Riemannian metric $g$ on $\Gamma$ which we consider as an analog of a  K\"{a}hler form. 
This metric is unrelated to the metric structure on $\Gamma,$ i.e., to the length of edges, but if we take $g$ to be trivial $$g = d'x\wedge d'' x \simeq  dx \otimes dx,$$
then the length $l(e)$ of an edge $e,$ the function $l(e)$ is a part of initial data for the metric graph $\Gamma,$ coincides with the length with respect to the Riemannian metric $g.$ In this case our theory is practically identical to the standard quantum graph theory, at least if we work with functions only. 

Also, in our case we require that all infinite-length edges should have finite length with respect to the Riemannian metric $g,$ otherwise we would get a different behavior of harmonic forms. 
For example, constant functions are harmonic on $\Gamma$, but if there is an infinite-length edge and $g$ is the trivial metric, a non-zero constant function does not belong $L^2$ since it has infinite norm. So if there are infinite-length edges, than the trivial Riemannian metric is not a viable option for our purpose.
\end{remark}

\begin{remark}
It would be interesting to study spectral properties (like asymptotics of eigenvalues, Weyl law, and so on) of the Laplace-Beltrami operator and compare them with spectral properties of complex curves and quantum graphs.
\end{remark}


\begin{thebibliography}{99}
\bibitem{BSU} Y.M. Berezansky, Z.G. Sheftel, G.F. Us, Functional analysis, vol. II,  Birkh\"{a}user Verlag, Basel, 1996. xvi+293 pp.
\bibitem{BK} Gregory Berkolaiko, Peter Kuchment; Introduction to quantum graphs. Mathematical Surveys and Monographs, 186. American Mathematical Society, Providence, RI, 2013. xiv+270 pp.
\bibitem{Dem} Jean-Pierre Demailly, Complex analytic and differential geometry, a draft version available at \url{https://www-fourier.ujf-grenoble.fr/~demailly/manuscripts/agbook.pdf}
\bibitem{GH}  Phillip Griffiths, Joseph Harris;  Principles of algebraic geometry. Pure and Applied Mathematics. Wiley-Interscience, New York, 1978. xii+813 pp.
\bibitem{GJR} Walter Gubler, Philipp Jell, Joseph Rabinoff, Dolbeault Cohomology of Graphs and Berkovich Curves, 2021, preprint arXiv:2111.05747.
\bibitem{IKMZ} Ilia Itenberg,  Ludmil Katzarkov, Grigory Mikhalkin, Ilia Zharkov; Tropical homology. Math. Ann. 374 (2019), no. 1-2, 963–1006.
\bibitem{JSS} Philipp Jell, Kristin Shaw, Jascha Smacka; Superforms, tropical cohomology, and Poincar\'{e} duality. Adv. Geom. 19 (2019), no. 1, 101–130
\bibitem{Lag} Aron Lagerberg;  Super currents and tropical geometry. Math. Z. 270 (2012), no. 3-4, 1011-1050.
\bibitem{Lan} Lionel Lang, Harmonic tropical morphisms and approximation. Math. Ann. 377 (2020), no. 1-2, 379–419.
\bibitem{Tel} Nicolae Teleman;  Combinatorial Hodge theory and signature operator. Invent. Math. 61 (1980), no. 3, 227–249.
Dolbeault Cohomology of Graphs and Berkovich Curves
Walter Gubler, Philipp Jell, Joseph Rabinoff, Dolbeault Cohomology of Graphs and Berkovich Curves, 2021, preprint arXiv:2111.05747.

\end{thebibliography}
\end{document}